\documentclass[11pt]{amsart}

\usepackage{amsmath, amscd, amsthm, amssymb, graphics, mathrsfs, setspace,fancyhdr,geometry,tabularx,shapepar, hyperref, xcolor, tikz, cite, booktabs, dsfont,amsfonts, marvosym, amsbsy, bm, tikz, array, mathtools,arydshln,tikz-cd}
\usetikzlibrary{matrix,arrows,backgrounds}
\usepackage[all]{xy}
\usepackage{leftidx}

\usepackage{mathrsfs}
\usepackage{wrapfig}
\usepackage{listings}

\usepackage[OT2,T1]{fontenc}
\DeclareSymbolFont{cyrletters}{OT2}{wncyr}{m}{n}
\DeclareMathSymbol{\Sha}{\mathalpha}{cyrletters}{"58}
\DeclareMathSymbol{\Zhe}{\mathalpha}{cyrletters}{17}
\usepackage{wasysym}

\usepackage{apptools}
\usepackage{chngcntr}

\mathchardef\hyph="2D

\hypersetup{colorlinks=false}
\theoremstyle{plain}
\newtheorem{theorem}{Theorem}[section]

\newtheorem{lemma}[theorem]{Lemma}

\newtheorem{proposition}[theorem]{Proposition}

\newtheorem{conjecture}[theorem]{Conjecture}
\newtheorem*{theorem*}{Theorem}
\newtheorem*{problem*}{Problem}
\newtheorem*{question*}{Question}
\newtheorem{theoremintro}{Theorem}

\theoremstyle{definition}
\newtheorem{definition}[theorem]{Definition}

\newtheorem{remark}[theorem]{Remark}

\newtheorem{ex}[theorem]{Example}

\AtAppendix{\counterwithin{theorem}{subsection}}

\numberwithin{equation}{section}

\newcommand{\colim}{\operatorname{colim}}

\newcommand{\Spec}{\operatorname{Spec}}

\newcommand{\Z}{{\mathbb Z}}
\newcommand{\R}{{\mathbb R}}

\newcommand{\A}{\mathbb{A}}
\newcommand{\G}{\mathbb{G}}

\newcommand{\im}{\operatorname{Im}}

\newcommand{\op}[1]{\operatorname{#1}}

\newcommand{\id}{\operatorname{id}}

\newcommand{\gm}{\mathbb{G}_{m}}

\newcommand{\Br}{\operatorname{Br}}
\newcommand{\GL}{\operatorname{GL}}
\newcommand{\PGL}{\operatorname{PGL}}
\newcommand{\Hom}{\operatorname{Hom}}

\newcommand{\Aut}{\operatorname{Aut}}
\newcommand{\Pic}{\operatorname{Pic}}
\newcommand{\Ext}{\operatorname{Ext}}

\newcommand{\bbP}{\mathbb{P}}

\newcommand{\leftexp}[2]{{\vphantom{#2}}^{#1}{#2}}
\newcommand{\weezer}{\leftexp{=}{\kern-0.23em\mathsf{W}}^{\kern-0.21em =}}
\newcommand*{\sheafhom}{\mathrm{H}\kern -.5pt om} 
\newcommand{\cd}[1]{\mathscr{D}(#1)}
\newcommand{\twist}[2]{{}^{#2} \kern -2.5pt #1}
\newcommand{\ICP}{\Zhe}

\begin{document}

\title{Separable algebras and coflasque resolutions}

\author[Ballard]{Matthew R Ballard}
\address{Department of Mathematics, University of South Carolina, 
Columbia, SC 29208}
\email{ballard@math.sc.edu}
\urladdr{\url{http://www.matthewrobertballard.com}}

\author[Duncan]{Alexander Duncan}
\address{Department of Mathematics, University of South Carolina, 
Columbia, SC 29208}
\email{duncan@math.sc.edu}
\urladdr{\url{http://people.math.sc.edu/duncan/}}

\author[Lamarche]{Alicia Lamarche}
\address{Department of Mathematics, University of Utah, 
Salt Lake City, UT 84112}
\email{lamarchr@math.utah.edu}
\urladdr{\url{http://alicia.lamarche.xyz}}

\author[McFaddin]{Patrick K. McFaddin}
\address{Department of Mathematics, Fordham University, 
New York, NY 10023}
\email{pkmcfaddin@gmail.com}
\urladdr{\url{http://mcfaddin.github.io/}}

\begin{abstract}
Over a non-closed field, it is a common strategy to use separable
algebras as invariants to distinguish algebraic and geometric objects.
The most famous example is the deep connection between Severi-Brauer varieties
and central simple algebras.
For more general varieties, one might use endomorphism algebras of line bundles,
of indecomposable vector bundles, or of exceptional objects in their
derived categories.

Using Galois cohomology, we describe a new invariant of reductive algebraic groups
that captures precisely when this strategy will fail.
Our main result characterizes this invariant in terms of coflasque
resolutions of linear algebraic groups introduced by
Colliot-Th\'el\`ene.
We determine whether or not this invariant is trivial for many fields.
For number fields, we show it agrees with the Tate-Shafarevich group of
the linear algebraic group, up to behavior at real places.
\end{abstract}

\maketitle

%%%%%%%%%%%%%%%%%%%%%%%%%%%%%%%%%%%%%%%%%%%%%%%%%%%%%%%%%%%%%%%%
%%%%%%%%%%%%%%%%%%%%%%%%%%%%%%%%%%%%%%%%%%%%%%%%%%%%%%%%%%%%%%%%
%%%%%%%%%%%%%%%%%%%%%%%%%%%%%%%%%%%%%%%%%%%%%%%%%%%%%%%%%%%%%%%%
% INTRODUCTION %
%%%%%%%%%%%%%%%%%%%%%%%%%%%%%%%%%%%%%%%%%%%%%%%%%%%%%%%%%%%%%%%%
%%%%%%%%%%%%%%%%%%%%%%%%%%%%%%%%%%%%%%%%%%%%%%%%%%%%%%%%%%%%%%%%
%%%%%%%%%%%%%%%%%%%%%%%%%%%%%%%%%%%%%%%%%%%%%%%%%%%%%%%%%%%%%%%%

\section{Introduction}
\label{sec:intro}

Given a base field $k$, an $n$-dimensional \emph{Severi-Brauer variety}
$X$ over $k$ is an (\'etale) $k$-form of the projective space $\bbP^n_k$; in other
words, there exists a finite separable field extension $L/k$ such that $X_L := X
\times_{\Spec(k)} \Spec(L)$ is isomorphic to $\bbP^n_L$.
The isomorphism classes of Severi-Brauer varieties of dimension $n$
are in bijective correspondence with central simple algebras $A$ of
degree $n+1$, which are forms of the algebra $M_{n+1}(k)$ of $(n+1)\times(n+1)$
matrices over $k$.

From another perspective, the isomorphism classes of
$n$-dimensional Severi-Brauer varieties over $k$ are classified by the
elements of the Galois cohomology set $H^1(k,\PGL_{n+1})$.
There is an injective function
\[
H^1(k,\PGL_n) \hookrightarrow \Br(k) = H^2(k,\gm),
\]
functorial with respect to the field $k$, which associates a given
Severi-Brauer variety to the Brauer equivalence class of the
corresponding central simple algebra.

A \emph{separable algebra} $A$ over a field $k$ is a direct sum
\[
A = \bigoplus_{i=1}^r M_{n_i}(D_i)
\]
of matrix algebras $M_{n_i}(D_i)$ where each $D_i$ is a
finite-dimensional division
$k$-algebra whose center is a separable field extension of $k$.
Alternatively, a separable algebra is an \'etale $k$-form of a direct sum of
matrix algebras over $k$.

The del Pezzo surfaces of degree $6$ are all $k$-forms of one another.
Blunk~\cite{Blunk} demonstrated how to associate a form of a separable $k$-algebra
to each del Pezzo surface of degree $6$ in such a way that two surfaces
are isomorphic if and only if their corresponding algebras are
isomorphic.  In this case, the split del Pezzo surface has the associated separable algebra $M_2(k)^{\oplus 3} \oplus M_3(k)^{\oplus 2}$.

Both Severi-Brauer varieties and del Pezzo surfaces are examples of
\emph{arithmetic toric varieties}: normal varieties which admit a faithful action of a torus (Definition~\ref{defn:torus}) with dense open orbit.
In \cite{Duncan}, it is shown that one can distinguish isomorphism
classes of $k$-forms of an arithmetic toric variety $X$ by separable $k$-algebras
whenever forms of $X$ with a rational point are retract rational.
In all these cases, the separable algebras are the direct sums of
endomorphism algebras of certain indecomposable vector bundles on the variety $X$.

It is natural to ask: can one can distinguish $k$-forms for wider classes
of objects via separable algebras?
For example, can varieties be distinguished by endomorphism algebras of
exceptional objects in their derived categories?
Are there even more exotic constructions?
The purpose of this paper is to precisely describe a fundamental obstruction to
all such strategies.

Recall that, under mild technical conditions,
the isomorphism classes of $k$-forms of an algebraic object $X$
are in bijection with the Galois cohomology set $H^1(k,G)$, where $G$ is the
automorphism group scheme of $X$.
If $A$ is a separable $k$-algebra with an algebraic action of an
algebraic group $G$, then
there is an algebraic group homomorphism $G \to \Aut(A)$.
We define
\begin{equation} \label{eq:ZheDefine}
\ICP(k,G) := \bigcap_A \ker
\left( H^1(k,G) \to H^1(k,\Aut(A)) \right)
\end{equation}
where the intersection runs over all separable $k$-algebras $A$
with a $G$-action.
Informally, $\ICP(k,G)$ is the set of $k$-forms of $X$ that cannot
be distinguished using forms of separable $k$-algebras. 

As a consequence of our main result, we can completely characterize 
this invariant using the theory
of flasque and coflasque resolutions of reductive algebraic groups, pioneered by
Colliot-Th\'el\`ene~\cite{CT2004,CT2008}, which is reviewed in
Section~\ref{sec:coflasque} below.

\begin{theoremintro}
  \label{thm:intro_main_cor}
Let $G$ be a (connected) reductive algebraic group over $k$.
Then
\[
\ICP(k,G) = \op{im} \left(H^1(k,C) \to H^1(k,G)\right),
\]
where
\[
1 \to P \to C \to G \to 1
\]
is an exact sequence of algebraic groups such that
\begin{enumerate}
\item $P$ is a quasitrivial torus and
\item $C$ is an extension of a coflasque torus by a semisimple simply
connected group.
\end{enumerate}
\end{theoremintro}

We expand Theorem~\ref{thm:intro_main_cor} into 
Theorem~\ref{thm:big_equivalence} which lists five equivalent
characterizations of $\ICP(k,G)$.

It turns out that $\ICP(k,G)$ cannot be detected by a more general class
of probes than separable algebras, which is a priori larger.
As is well-known, the connecting homomorphism
\begin{displaymath} 
  H^1(k,\op{PGL}_1(A)) \to H^2(k,\op{GL}_1(Z(A)))
\end{displaymath}
is injective. Thus, if the composition
\begin{equation} \label{eq:intro_coh_inv}
  H^1(k,G) \to H^2(k,\op{GL}_1(Z(A))) 
\end{equation}
vanishes, then original map must have been trivial already. This leads one
to investigate a notion that captures Equation~\eqref{eq:intro_coh_inv}. 

\subsection{Cohomological invariants}
\label{sec:intro_CI}

Recall that the Galois cohomology pointed set $H^i(k,G)$ is functorial
in both $G$ and $k$.  In particular, fixing $G$, we may view $H^i(-,G)$
as a functor from the category of field extensions
of $k$ to the category of pointed sets
(or to groups, or to abelian groups, appropriately).
Let $\op{Inv}^2_\ast(G,S)$ be the group of \emph{normalized cohomological
invariants}, i.e., natural transformations
\[
\alpha : H^1(-,G) \to H^2(-,S)
\]
where $G$ is a linear algebraic group, $S$ is a commutative linear algebraic
group, and $\alpha$ takes the distinguished point to zero.

Recall that a linear algebraic group $G$ is \emph{special} 
if $H^1(K,G_K)$ is trivial for all field extensions $K/k$.
In Theorem~\ref{thm:big_equivalence} below, we will see that for
reductive algebraic groups $G$ we have many equivalent characterizations
of $\ICP(k,G)$. In particular,
\begin{equation} \label{eq:intro_icp_equals_coh_inv_kernels}
  \ICP(k,G) = \bigcap_S \bigcap_\alpha \op{ker}
  \left( \alpha(k) \colon H^1(k,G) \to H^2(k,S) \right)
\end{equation}
where the intersections run over all special tori $S$
and all normalized cohomological invariants $\alpha \in \op{Inv}^2_\ast(G,S)$.
In particular, for any finite separable field extension $F/k$, 
we have $H^2(k,S)\cong\Br(F)$ for some choice of special torus $S$. 

Thus, not only is $\ICP(k,G)$ an obstruction to differentiating
Brauer classes obtained from actions of $G$ on separable algebras,
but also to those obtained from almost completely arbitrary maps
(requiring only they behave well under field extensions). 

\subsection{Applications}
\label{sec:intro_apps}

Theorem~\ref{thm:intro_main_cor} allows us to compute $\ICP(k,G)$ in many cases of interest.
For example, the following consequences are discussed in
Section~\ref{sec:applications}:

\begin{itemize}
\item when $k$ is a finite field or
nonarchimedean local field $\ICP(k,G)$ is trivial.
\item if $S$ is a torus, the functor $\ICP(-,S)$ is
trivial if and only if $S$ is retract rational. 
\item if $S$ is a torus over a number field $k$, then
$$\Zhe(k,S)=\Sha^1(k,S).$$
\item if $G$ is semisimple and simply-connected over a number field, 
then $$\Zhe(k,G) = \prod_{v \text{ real}} \Zhe(k_v,G_v).$$ 
\item if $k$ is a totally imaginary number field, then
$$\Zhe(k,G)=\Sha^1(k,G).$$
\end{itemize}

Retract rationality will be recalled in Section~\ref{sec:coflasque} below (see Definition~\ref{defn:rationality})
and $\Sha^1(k,G)$ denotes the Tate-Shafarevich group,
which is discussed in Section~\ref{sec:applications}.
Indeed the notation $\Zhe$ was chosen to remind the reader of
$\Sha$.  The connection is made explicit for number fields in the
following:

\begin{theoremintro} \label{thm:Zhe_number_field}
Let $G$ be a reductive algebraic group over a number field $k$.
Then there exists a canonical isomorphism
\[
\ICP(k,G) = \Sha^1(k,G) \times \prod_{v \textrm{ real}} \ICP(k_v, G_{k_v}).
\]
\end{theoremintro}

Finally, as a further application of the results of this paper, the authors 
in \cite{patho} show how to use nontrivial elements of $\ICP(k,G)$ 
to twist varieties, possessing exceptional objects in their 
derived categories, by $G$-torsors without twisting the exceptional 
objects themselves. 

The structure of the remainder of the paper is as follows.
In Section~\ref{sec:coflasque}, we overview the theory of coflasque
and flasque resolutions, moving from lattices to tori and then treating general
reductive algebraic groups.
In Section~\ref{sec:big_thm}, we prove Theorem~\ref{thm:intro_main_cor} as well as
several other equivalent characterizations of $\Zhe(k,G)$.
Finally, in Section~\ref{sec:applications}, we compute $\Zhe(k,G)$ in
several special cases and establish Theorem~\ref{thm:Zhe_number_field}.
In an Appendix, we prove a generalization of a result of Blinstein and
Merkurjev used to prove the main theorem, which may be of independent
interest.

%%%%%%%%%%%%%%%%%%%%%%%%%%%%%%%%%%%%%%%%%%%%%%%%%%%%%%%%%%%%%%%%
%%%%%%%%%%%%%%%%%%%%%%%%%%%%%%%%%%%%%%%%%%%%%%%%%%%%%%%%%%%%%%%%
% SUBSECTION %
%%%%%%%%%%%%%%%%%%%%%%%%%%%%%%%%%%%%%%%%%%%%%%%%%%%%%%%%%%%%%%%%
%%%%%%%%%%%%%%%%%%%%%%%%%%%%%%%%%%%%%%%%%%%%%%%%%%%%%%%%%%%%%%%%

\subsection*{Acknowledgements}
\label{sec:hat_tip}

The authors would like to thank B.~Antieau, M.~Borovoi,
J.-L.~Colliot-Th\'el\`ene, P.~Gille,
A.~Merkurjev, and anonymous referees for helpful comments. 
Via the first author, this material is based upon work supported by the National
Science Foundation under Grant No.~NSF DMS-1501813.
Via the second author, this work was supported by a grant from the Simons Foundation
(638961, AD).
The third author was partially supported by a USC SPARC grant.
The fourth author was partially supported by an AMS-Simons travel grant.

%%%%%%%%%%%%%%%%%%%%%%%%%%%%%%%%%%%%%%%%%%%%%%%%%%%%%%%%%%%%%%%%
%%%%%%%%%%%%%%%%%%%%%%%%%%%%%%%%%%%%%%%%%%%%%%%%%%%%%%%%%%%%%%%%
% SUBSECTION %
%%%%%%%%%%%%%%%%%%%%%%%%%%%%%%%%%%%%%%%%%%%%%%%%%%%%%%%%%%%%%%%%
%%%%%%%%%%%%%%%%%%%%%%%%%%%%%%%%%%%%%%%%%%%%%%%%%%%%%%%%%%%%%%%%

\subsection*{Notation and Conventions}
\label{sec:notation}
Throughout, $k$ denotes an arbitrary field with separable closure $\overline{k}$.
Let $\Gamma_k$ denote the absolute Galois group $\op{Gal}(\overline{k}/k)$,
which is a profinite group.
A variety is an integral separated scheme of finite type over a field.
A linear algebraic group is a smooth affine group scheme of finite type over $k$.
A reductive group is assumed to be connected.

Let $\pi : \Spec(L) \to \Spec(k)$ be the morphism associated to a
separable field extension $L/k$.
For a $k$-variety $X$, we write
$X_L : = X \times_{\Spec L} \Spec k = \pi^\ast(X)$
and $\overline{X} : = X_{\overline{k}}$.
For an $L$-variety $Y$, we write
$R_{L/k}(Y) := \pi_\ast(Y)$ for the Weil restriction, which is a
$k$-variety.

Let $\GL_n$ denote the general linear group scheme and
$\gm = \Spec(\Z[t^{\pm 1}])=\GL_1$ as the multiplicative group over $\Z$.
We will simply write $\GL_n$ for $ \GL_{n,k}$ or $\gm$ for $\G_{m,k}$
when there is no danger of confusion.
Unless otherwise specified, a $G$-torsor is a \emph{right} $G$-torsor.

We will reference the following categories:
\begin{itemize}
\item $\mathsf{Set}$ is the category of sets.
\item $\mathsf{Set}_\ast$ is the category of pointed sets.
\item $\mathsf{Grp}$ is the category of groups.
\item $\mathsf{Ab}$ is the category of abelian groups.
\item $\mathsf{Lat}$ is the category of finitely-generated free abelian
groups.
\end{itemize}
Given a base field $k$:
\begin{itemize}
\item $k\hyph\mathsf{Alg}$ is the category of associative $k$-algebras.
\item $k\hyph\mathsf{Fld}$ is the category of field extensions of $k$.
\item $k\hyph\mathsf{Grp}$ is the category of algebraic groups over $k$.
\end{itemize}

Given a profinite group $\Gamma$ and a concrete category $\mathsf{C}$
(in other words, $\mathsf{C}$ is equipped with a faithful functor to
category of sets),
we write $\Gamma\hyph\mathsf{C}$ to denote the category of objects
whose underlying sets are endowed with the discrete topology
and a continuous left action of $\Gamma$.
Objects in $\Gamma\hyph\mathsf{Set}$,
$\Gamma\hyph\mathsf{Ab}$,
and $\Gamma\hyph\mathsf{Lat}$
are called $\Gamma$-sets, $\Gamma$-modules, and $\Gamma$-lattices
respectively.

For $\Gamma$-modules $A,B$, we use the shorthand notation
$\Hom_\Gamma(A,B) := \Hom_{\Gamma\hyph\mathsf{Ab}}(A,B)$
and $\Ext^i_\Gamma(A,B) := \Ext^i_{\Gamma\hyph\mathsf{Ab}}(A,B)$.
For linear algebraic groups $A, B$ over $k$ with $B$ commutative,
we denote by $\Ext^1_k(A,B)$ the group of isomorphism classes of
central extensions of algebraic groups
\[
1 \to B \to G \to A \to 1
\]
under the usual Baer sum.

For $k$-algebras $A$ and $B$, we use the shorthand
$\Hom_k(A,B) := \Hom_{k\hyph\mathsf{Alg}}(A,B)$. 
For algebraic groups $A$ and $B$ defined over $k$, we use the shorthand
$\Hom_k(A,B) := \Hom_{k\hyph\mathsf{Grp}}(A,B)$.
For a scheme $X$ and an \'etale sheaf $\mathcal{F}$ on $X$,
we write $H^n(X,\mathcal{F})$ to denote \'etale cohomology.
In particular, we write $\Pic(X)=H^1(X,\gm)$ and $\Br(X)=H^2(X,\gm)$.
For a field $k$, we write
$H^n(k,\mathcal{F}):=H^n(\Spec(k),\mathcal{F})$.
For a profinite group $\Gamma$ and a (continuous) $\Gamma$-set $A$,
we write $H^n(\Gamma,A)$ for the appropriate cohomology set,
assuming this makes sense given $n$ and $A$.

For readers unfamiliar with the Cyrillic letter $\ICP$, ``Zhe'', it 
is pronounced close to the ``s'' in ``treasure''. 

%%%%%%%%%%%%%%%%%%%%%%%%%%%%%%%%%%%%%%%%%%%%%%%%%%%%%%%%%%%%%%%%
%%%%%%%%%%%%%%%%%%%%%%%%%%%%%%%%%%%%%%%%%%%%%%%%%%%%%%%%%%%%%%%%
%%%%%%%%%%%%%%%%%%%%%%%%%%%%%%%%%%%%%%%%%%%%%%%%%%%%%%%%%%%%%%%%
% COFLASQUE RESOLUTIONS %
%%%%%%%%%%%%%%%%%%%%%%%%%%%%%%%%%%%%%%%%%%%%%%%%%%%%%%%%%%%%%%%%
%%%%%%%%%%%%%%%%%%%%%%%%%%%%%%%%%%%%%%%%%%%%%%%%%%%%%%%%%%%%%%%%
%%%%%%%%%%%%%%%%%%%%%%%%%%%%%%%%%%%%%%%%%%%%%%%%%%%%%%%%%%%%%%%%

\section{Coflasque resolutions}
\label{sec:coflasque}

%%%%%%%%%%%%%%%%%%%%%%%%%%%%%%%%%%%%%%%%%%%%%%%%%%%%%%%%%%%%%%%%
%%%%%%%%%%%%%%%%%%%%%%%%%%%%%%%%%%%%%%%%%%%%%%%%%%%%%%%%%%%%%%%%
% SUBSECTION %
%%%%%%%%%%%%%%%%%%%%%%%%%%%%%%%%%%%%%%%%%%%%%%%%%%%%%%%%%%%%%%%%
%%%%%%%%%%%%%%%%%%%%%%%%%%%%%%%%%%%%%%%%%%%%%%%%%%%%%%%%%%%%%%%%

\subsection{Preliminaries on lattices}
\label{sec:lattices}

We recall some facts about $\Gamma$-lattices following \cite{CTS77}, 
see also \cite{Vosky}.

\begin{definition}
Let $\Gamma$ be a profinite group and let $M$ be a $\Gamma$-lattice.
Note that the image of the $\Gamma$-action factors through a finite group $G$
called the \emph{decomposition group}, which acts faithfully on $M$.
\begin{enumerate}
\item $M$ is \emph{permutation} if there is a $\Z$-basis of $M$ permuted
by $\Gamma$.
\item $M$ is \emph{stably permutation} if there exist permutation
lattices $P_1$ and $P_2$ such that $M \oplus P_1 = P_2$.
\item $M$ is \emph{invertible} if it is a direct summand of a permutation lattice.
\item $M$ is \emph{quasi-permutation} if there
exists a short exact sequence
\[
0 \to M \to P_1 \to P_2 \to 0
\]
where $P_1$ and $P_2$ are permutation lattices.
\end{enumerate}
\end{definition}

Given a $\Gamma$-lattice $M$, let $[M]$ denote its similarity class.
In other words, $[M_1]=[M_2]$ if and only if there exist permutation
$\Gamma$-lattices $P_1$ and $P_2$ such that
$M_1 \oplus P_1 \cong M_2 \oplus P_2$.
Observe that the set of similarity classes form a monoid under direct
sum. Being stably permutation amounts to saying that $[M]=[0]$,
while being invertible amounts to saying there exists a lattice $L$ such
that $[M]+[L]=[0]$.

Given a $\Gamma$-lattice $M$, the \emph{dual lattice} 
$M^\vee := \Hom_{\mathsf{Ab}}(M,\Z)$ is the set of group homomorphisms from
$M$ to $\Z$ with the natural $\Gamma$-action
where $\Z$ has the trivial $\Gamma$-action.
Note that this duality induces an exact anti-equivalence of the category
of $\Gamma$-lattices with itself.

\begin{definition}
Let $M$ be a $\Gamma$-lattice.
\begin{enumerate}
\item $M$ is \emph{coflasque} if $H^1(\Gamma',M)=0$ for all
open subgroups $\Gamma' \subseteq \Gamma$.
\item $M$ is \emph{flasque} if $M^\vee$ is coflasque.
\item A \emph{flasque resolution of $M$ of the first type} is an exact sequence
\[
0 \to M \to P \to F \to 0
\]
while a \emph{flasque resolution of $M$ of the second type} is an exact sequence
\[
0 \to P \to F \to M \to 0
\]
where, in each case, $P$ is a permutation lattice and $F$ is a flasque lattice.
\item A \emph{coflasque resolution of $M$ of the first type} is an exact sequence
\[
0 \to C \to P \to M \to 0
\]
while a \emph{coflasque resolution of $M$ of the second type} is an exact sequence
\[
0 \to M \to C \to P \to 0
\]
where, in each case, $P$ is a permutation lattice and $C$ is a coflasque lattice.
\end{enumerate}
\end{definition}

The following alternative characterizations of flasque, coflasque, and
invertible will be useful:

\begin{lemma}
\label{lemma:many_faces_flasque}
Let $\Gamma$ be a profinite group.
\begin{enumerate}
	\item The following are equivalent for a $\Gamma$-module $C$:
		\begin{itemize}
			\item $C$ is coflasque.
			\item $\Ext_\Gamma^1(P,C)=0$ for every
permutation $\Gamma$-lattice $P$.
			\item $\Ext_\Gamma^1(Q,C)=0$ for every
invertible $\Gamma$-lattice $Q$.\\
		\end{itemize}
	\item The following are equivalent for a $\Gamma$-module $F$:
		\begin{itemize}
			\item $F$ is flasque.
			\item $\Ext_\Gamma^1(F,P)=0$
for every permutation $\Gamma$-lattice $P$.
			\item $\Ext_\Gamma^1(F,Q)=0$
for every invertible $\Gamma$-lattice $Q$.\\
		\end{itemize}
	\item The following are equivalence for a $\Gamma$-module $M$:
		\begin{itemize}
			\item $M$ is invertible.
			\item $\Ext_\Gamma^1(M,C)=0$
for every coflasque $\Gamma$-lattice $C$.
			\item $\Ext_\Gamma^1(F,M)=0$
for every flasque $\Gamma$-lattice $F$.
		\end{itemize}	
\end{enumerate}
\end{lemma}

\begin{proof}
This is standard.
See, e.g., \cite[Lemme 9]{CTS77} and \cite[0.5]{ColSan87Principal}.
\end{proof}

Flasque/coflasque resolutions of both types always exist
but are never unique;
however, the similarity classes $[F]$ and $[C]$ are well-defined
\cite[Lemma 0.6]{ColSan87Principal}.

It is well known that flasque and coflasque resolutions of the first
type are
``versal'' in the following sense:

\begin{lemma} \label{lem:versal_first_type}
Let $M$ be a $\Gamma$-lattice.
If
\[
0 \to C \to P \xrightarrow{\alpha} M \to 0
\]
is a coflasque resolution of the first type,
then any morphism $P' \to M$ with $P'$ invertible factors through $\alpha$.
Dually, if
\[
0 \to M \xrightarrow{\beta} P \to  F \to 0
\]
is a flasque resolution of the first type,
then any morphism $M \to P'$ with $P'$ invertible factors through $\beta$.
\end{lemma}

\begin{proof}
See \cite[Lemma 1.4]{CTS77}.
\end{proof}

Less well known is that resolutions of the second
type also satisfy a ``versality'' property.

\begin{lemma} \label{lem:versal_second_type}
Let $M$ be a $\Gamma$-lattice.
Suppose
\[
0 \to M \xrightarrow{\alpha} C \to P \to 0
\]
is a coflasque resolution of the second type and
\[
0 \to M \xrightarrow{\gamma} N \to Q \to 0
\]
is an extension of $\Gamma$-lattices with $Q$ invertible.
Then there is a morphism $\phi:N \to C$ such that
$\phi \circ \gamma = \alpha$.

Let $M$ be a $\Gamma$-lattice.
Suppose
\[
0 \to P \to F \xrightarrow{\alpha} M \to 0
\]
is a flasque resolution of the second type and
\[
0 \to Q \to N \xrightarrow{\gamma} M \to 0
\]
is an extension of $\Gamma$-lattices with $Q$ invertible.
Then there is a morphism $\phi:F \to N$ such that
$\gamma \circ \phi = \alpha$.
\end{lemma}

\begin{proof}
From Lemma~\ref{lemma:many_faces_flasque}, an equivalent condition that $C$ is coflasque is that
$\Ext^1_{\Gamma}(Q,C)=0$ for all invertible modules $Q$.
Thus from the exact sequence
\[
\Hom_\Gamma(Q,P) \to \Ext^1_{\Gamma}(Q,M) \to \Ext^1_{\Gamma}(Q,C) = 0
\]
there exists some map $\beta: Q \to P$ such that the extension 
\begin{displaymath}
  0 \to M \to Q \oplus_P C \to Q \to 0
\end{displaymath}
is isomorphic to 
\begin{displaymath}
  0 \to M \to N \to Q \to 0.
\end{displaymath}
The desired homomorphism $\phi: N \to C$ is the composition 
\begin{displaymath}
  N \cong Q \oplus_P C \to C. 
\end{displaymath}

The result for flasque resolutions follows by duality.
\end{proof}

%%%%%%%%%%%%%%%%%%%%%%%%%%%%%%%%%%%%%%%%%%%%%%%%%%%%%%%%%%%%%%%%
%%%%%%%%%%%%%%%%%%%%%%%%%%%%%%%%%%%%%%%%%%%%%%%%%%%%%%%%%%%%%%%%
% SUBSECTION %
%%%%%%%%%%%%%%%%%%%%%%%%%%%%%%%%%%%%%%%%%%%%%%%%%%%%%%%%%%%%%%%%
%%%%%%%%%%%%%%%%%%%%%%%%%%%%%%%%%%%%%%%%%%%%%%%%%%%%%%%%%%%%%%%%

\subsection{Preliminaries on algebraic tori}
\label{sec:tori}

\begin{definition}\label{defn:torus}
A $k$-\emph{torus} is an algebraic group $T$ over $k$ such that
$T _{\overline{k}} \cong \G_{m,\overline{k}} ^n$
for some non-negative integer $n$.
A torus is \emph{split} if $T \cong \G_{m,k}^n$.  A field extension
$L/k$ satisfying $T_L \cong \G_{m,L} ^n$ is called a \emph{splitting field}
of the torus $T$.
Any torus admits a finite Galois splitting field. 
\end{definition}

Recall that there is an anti-equivalence of categories between
$\Gamma_k$-lattices and $k$-tori, which we will call \emph{Cartier duality}
(see, e.g., \cite{Vosky}).
Given a torus $T$, the Cartier dual (or \emph{character lattice}) $\widehat{T}$ is the
$\Gamma$-lattice $\op{Hom}_{\bar{k}}(\overline{T},
\mathbb{G}_{m,\bar{k}})$.
Given a $\Gamma_k$-lattice $M$, we use $\cd M$ to denote the
Cartier dual torus.

\begin{definition}
Let $T$ be a torus with corresponding $\Gamma_k$-lattice $M := \widehat{T}$.
\begin{enumerate}
\item $T$ is \emph{quasi-trivial} if $M$ is permutation.
\item $T$ is \emph{flasque} if $M$ is flasque.
\item $T$ is \emph{coflasque} if $M$ is coflasque.
\end{enumerate}
Similarly, we may define flasque/coflasque resolutions
of both types via Cartier duality.
\end{definition}

\begin{proposition} \label{prop:invertible_H1_vanish}
A torus $T$ is special if and only if $\widehat{T}$ is an invertible
$\Gamma_k$-lattice.
\end{proposition}

\begin{proof}
This follows from the classification of
special tori due to Colliot-Th\'el\`ene
\cite[Theorem 13]{Huruguen}.
\end{proof}

As in the introduction, a \emph{separable algebra} $A$ over $k$ is a
finite direct sum of finite-dimensional matrix algebras
over finite-dimensional division $k$-algebras whose centers are separable
field extensions over $k$.
Given a separable algebra $A$ over $k$, we recall that $\GL_1(A)$ is the
group scheme of units of $A$, i.e., 
\[
\GL_1(A)(R) := (A \otimes_k R)^\times 
\]
for any commutative $k$-algebra $R$.

An \'etale algebra over $k$ of degree $n$ is a commutative separable
algebra over $k$ of dimension $n$.
In other words, $E = F_1 \times \cdots \times F_r$
where $F_1, \ldots, F_r$ are separable field extensions of $k$.
There is an antiequivalence between finite $\Gamma_k$-sets $\Omega$ and
\'etale algebras $E$ via
\[
\Omega = \Hom_{k\hyph\mathsf{Alg}}(E,\bar{k})
\textrm{ and } E = \Hom_{\Gamma_k\hyph\mathsf{Set}}(\Omega,\bar{k})
\]
with the natural $\Gamma_k$-action and $k$-algebra structure
on $\bar{k}$ (see, e.g., \cite[\S{18}]{BOI}).

\begin{proposition} \label{prop:hilbert90_shapiro}
Let $E = F_1 \times \cdots \times F_r$
be an \'etale algebra over $k$ of degree $n$,
where $F_1, \ldots, F_r$ are separable field extensions of $k$.
Let $T = R_{E/k} \gm$ be the Weil restriction
and let $\Omega := \Hom_k(E,\bar{k})$ be the corresponding $\Gamma$-set.
\begin{enumerate}
\item $T(k) = E^\times$.
\item $\widehat{T}$ is a permutation $\Gamma_k$-lattice with a canonical basis
isomorphic to $\Omega$.
\item $H^1(k,T) = 1$.
\item $H^2(k,T) = \prod_{i=1}^r \Br(F_i)$.
\end{enumerate}
\end{proposition}

\begin{proof}
These are standard consequences of Hilbert's Theorem 90 and
Shapiro's Lemma.
\end{proof}

Let us now recall some relevant rationality properties.

\begin{definition}\label{defn:rationality} A $k$-variety $X$ is \emph{rational} if $X$ is birationally
equivalent to $\A^n_k$ for some $n \ge 0$.
We say $X$ is \emph{stably rational} if $X \times \A^n_k$ is birational to
$\A^m_k$ for some $n,m \ge 0$.
We say $X$ is \emph{retract rational} if there is a dominant rational
map $f : \A^n_k \dasharrow X$ that has a rational section
$s : X \dasharrow \A^n_k$ such that $f \circ s$ is the identity on $X$. 
\end{definition}

A complete characterization of rationality of tori is still an open
problem (it is not known if all stably rational tori are
rational).
However, stable rationality and retract rationality of a torus is
completely understood via its flasque resolutions.

\begin{theorem} \label{theorem:torus_rationality}
Let $T$ be a $k$-torus and
\[
1 \to F \to P \to T \to 1
\]
a flasque resolution of the first type.
\begin{itemize}
\item $T$ is stably rational if and only if $\widehat{F}$ is stably
permutation.
\item $T$ is retract rational if and only if $\widehat{F}$ is invertible.
\end{itemize}
\end{theorem}

\begin{proof}
The first item is \cite[Theorem 2]{VoskyStable}.
The second is \cite[Theorem 3.14]{SaltmanRR}.
\end{proof}

%%%%%%%%%%%%%%%%%%%%%%%%%%%%%%%%%%%%%%%%%%%%%%%%%%%%%%%%%%%%%%%%
%%%%%%%%%%%%%%%%%%%%%%%%%%%%%%%%%%%%%%%%%%%%%%%%%%%%%%%%%%%%%%%%
% SUBSECTION %
%%%%%%%%%%%%%%%%%%%%%%%%%%%%%%%%%%%%%%%%%%%%%%%%%%%%%%%%%%%%%%%%
%%%%%%%%%%%%%%%%%%%%%%%%%%%%%%%%%%%%%%%%%%%%%%%%%%%%%%%%%%%%%%%%

\subsection{Flasque and coflasque resolutions of algebraic groups}
\label{sec:FC_alg_grp}

We recall how one can define flasque and coflasque resolutions for
more linear algebraic groups following \cite{CT2008}.

Let $G$ be a (connected) reductive algebraic group over a field $k$.
Note that since our main application will be understanding the first
Galois cohomology set of $G$, in characteristic $0$ the reductive
hypothesis is largely harmless.
Let $G^{ss}$ be the derived subgroup of $G$, which is semisimple,
and let $G^{tor}$ be the quotient $G/G^{ss}$, which is a torus.

\begin{definition}
Let $G$ be a reductive algebraic group.
\begin{itemize}
\item
The group $G$ is \emph{quasi-trivial} if $G^{tor}$ is a quasi-trivial
torus and $G^{ss}$ is simply-connected.
\item
The group $G$ is \emph{coflasque} if $G^{tor}$ is a coflasque
torus and $G^{ss}$ is simply-connected.
\item
A \emph{flasque resolution} of $G$ is a short exact sequence
\[
1 \to S \to H \to G \to 1
\]
where $S$ is a flasque torus and $H$ is quasi-trivial.
\item
A \emph{coflasque resolution} of $G$ is a short
exact sequence
\[
1 \to P \to C \to G \to 1
\]
where $P$ is a quasi-trivial torus and $C$ is coflasque.
\end{itemize}
\end{definition}

The group extensions in a flasque or coflasque resolution are
automatically central
since $G$ cannot act non-trivially on a torus.
Indeed, the automorphism group scheme of a torus $T$ has trivial
connected component and we assume the group $G$ is connected

Unlike the situation for $\Gamma$-lattices and tori,
the symmetry between flasque and coflasque is now broken.
In the case where $G$ is a torus, the resolutions above specialize to
flasque resolutions of the first type and coflasque resolutions of the
second type.
In this context, we do not refer to the ``type'' of a
flasque or coflasque resolution.
However, as for tori, the flasque and coflasque resolutions defined
above always exist:

\begin{theorem}[Colliot-Th\'el\`ene] \label{thm:CT_collected}
 Let $G$ be a reductive algebraic group over $k$. Then there exists both a flasque resolution and coflasque resolution of $G$. Moreover, for any two coflasque resolutions
 \begin{gather*}
  1 \to P_1 \to C_1 \to G \to 1 \\
  1 \to P_2 \to C_2 \to G \to 1
 \end{gather*}
 there is an isomorphism
 \begin{displaymath}
   P_1 \times C_2 \cong P_2 \times C_1.
 \end{displaymath}
\end{theorem}

\begin{proof}
 The existence statements are \cite[Proposition 3.1]{CT2008} and \cite[Proposition 4.1]{CT2008}. The isomorphism is \cite[Proposition 4.2(i)]{CT2008}. 
\end{proof}

\begin{proposition} \label{prop:coflasque_general_versal}
Suppose $G$ is a reductive algebraic group and
consider a coflasque resolution
\[
1 \to P \to C \to G \to 1
\]
where $P$ is a quasi-trivial torus and $C$ is coflasque.
Suppose there exists an extension
\[
1 \to S \to H \to G \to 1
\] 
where $S$ is a central special torus.
Then there exists a morphism $C \to H$ inducing
a morphism of the extensions above that is the identity on $G$.
\end{proposition}

\begin{proof}
This proof is a variation of that of Proposition~4.2~of~\cite{CT2008}.
Let $E$ be the fiber product of $H$ and $C$ over $G$.
We have a commutative diagram
\[
\xymatrix{
& & 1 \ar[d] & 1 \ar[d] \\
& & P \ar@{=}[r] \ar[d] & P \ar[d] \\
1 \ar[r] & S \ar[r] \ar@{=}[d] & E \ar[d] \ar[r] & C \ar[r] \ar[d] & 1 \\
1 \ar[r] & S \ar[r] & H \ar[r] \ar[d] & G \ar[r] \ar[d] & 1 \\
& & 1 & 1 & }
\]
with exact rows and columns.
From \cite[Proposition 1.10 and 2.6]{CT2008},
we know $H^1(C,Q)=0$ for $C$ coflasque and $Q$ a
quasi-trivial torus.
Since $S$ is special there is a factorization $S \to Q \to S$ of the
identity for some
quasi-trivial torus $Q$, and thus $H^1(C,S)=0$.
Arguing as in the proof of \cite[Proposition 3.2]{CT2008}
(or using Theorem~\ref{thm:CTtorsor_to_group} below),
we conclude the group extension
\[
1 \to S \to E \to C \to 1
\]
is split.
The composite morphism $C \to E \to H$ gives the desired result.
\end{proof}

\begin{proposition} \label{prop:coflasque_resolution_independence}
Given a reductive algebraic group $G$ and a coflasque resolution
\[
1 \to P \to C \to G \to 1 \ ,
\]
the natural morphism
\[
H^1(k,C) \to H^1(k,G)
\]
is injective and its image is independent of the choice of
coflasque resolution.
\end{proposition}

\begin{proof}
Since $P$ is central, the fibers of the natural morphism
$H^1(k,C) \to H^1(k,G)$ are either empty or are torsors under
$H^1(k,P)$.  Since $P$ is quasi-trivial, $H^1(k,P)$ is trivial
and we conclude that $H^1(k,C) \to H^1(k,G)$ is injective.

Suppose
\[
1 \to P' \to C' \to G \to 1 \ ,
\]
is another coflasque resolution of $G$.
From \cite[Proposition 4.2(i)]{CT2008} and its proof,
there is an isomorphism $\alpha : P \times C' \cong P' \times C$.
such that the diagram
\[
\xymatrix{
P \times C' \ar[d]^\alpha \ar[r] & C' \ar[r] & G \ar@{=}[d]\\
P' \times C \ar[r] & C \ar[r] & G}
\]
commutes.
As above, since $P$ is quasi-trivial, the projection $P \times C' \to C'$
induces an isomorphism $H^1(k,P \times C') \cong H^1(k,C')$.
Thus the composite
\[
H^1(k,C') \to H^1(k,P \times C') \xrightarrow{\alpha}
H^1(k,P' \times C) \to H^1(k,C)
\]
is an isomorphism and induces equality of the images in $H^1(k,G)$.
\end{proof}

Note that a flasque resolution of a reductive algebraic group $G$ does not in
general give rise to a flasque resolution of its abelianization.
However, this does occur if $G$ is coflasque:

\begin{proposition} \label{prop:flasqueOfCoflasque}
If $C$ is a coflasque reductive algebraic group, then any flasque resolution of
$C$ gives rise to a commutative diagram
\[
\xymatrix{
1 \ar[r] &
S \ar[r] \ar@{=}[d] &
H \ar[r] \ar[d] &
C \ar[r] \ar[d] &
1 \\
1 \ar[r] &
S \ar[r] &
H^{tor} \ar[r] &
C^{tor} \ar[r] &
1 }
\]
with exact rows, where $H$ is a quasi-trivial algebraic group,
$S$ is a flasque torus, and the vertical maps are abelianizations.
Note that both rows are flasque resolutions. 
\end{proposition}

\begin{proof}
The only potential problem is that abelianization is not left-exact in
general.
The morphism $\varphi: H \to C$ induces a surjective morphism $H' \to C'$ of
their derived subgroups with commutative kernel $H' \cap S$.
However, since $C$ is coflasque, the semisimple algebraic group
$C'$ is simply-connected by definition.
Thus $H' \to C'$ is an isomorphism and $H' \cap S = 1$.
Consider the map $S \to \varphi^{-1}(C')/H'$. 
The kernel is $S \cap H' = 1$. Given $h \in \varphi^{-1}(C')$, 
since $\varphi|_{H'}$ is an isomorphism, there is some $h'$ with
$\varphi(hh') = 1$ so $S \to \varphi^{-1}(C')/H'$ is surjective.
We conclude that 
\begin{displaymath}
  1 \to S \to H/H' \to C/C' \to 1
\end{displaymath}
is exact. 
\end{proof}

%%%%%%%%%%%%%%%%%%%%%%%%%%%%%%%%%%%%%%%%%%%%%%%%%%%%%%%%%%%%%%%%
%%%%%%%%%%%%%%%%%%%%%%%%%%%%%%%%%%%%%%%%%%%%%%%%%%%%%%%%%%%%%%%%
% FAYGO %
%%%%%%%%%%%%%%%%%%%%%%%%%%%%%%%%%%%%%%%%%%%%%%%%%%%%%%%%%%%%%%%%
%%%%%%%%%%%%%%%%%%%%%%%%%%%%%%%%%%%%%%%%%%%%%%%%%%%%%%%%%%%%%%%%

\section{Cohomological Invariants}
\label{sec:big_thm}

We review the notion of a \emph{cohomological invariant}
following \cite{Skip}.
Fix a base field $k$ and recall our notation $k\hyph\mathsf{Fld}$
for the category of field extensions of $k$.
We consider two functors
\[
A : k\hyph\mathsf{Fld} \to \mathsf{Sets}_\ast
\]
and
\[
H : k\hyph\mathsf{Fld} \to \mathsf{Ab} \ .
\]
A \emph{normalized $H$-invariant of $A$} is a morphism of functors $A \to H$.
The group of all such invariants will be denoted $\op{Inv}_\ast(A,H)$.

\begin{remark}
We demand a priori that $A$ is a functor into \emph{pointed} sets.
This explains the adjective ``normalized.''
This condition is harmless as a general $H$-invariant of $A$
can be written uniquely as the sum of a normalized invariant
and a ``constant'' invariant coming from $H(k)$.
\end{remark}

The two kinds of functors we will consider are as follows.
Given an algebraic group $G$ over $k$,
we may view Galois cohomology
\[
H^i(-,G) : k\hyph\mathsf{Fld} \to \mathsf{Sets}_\ast
\]
as a functor (the codomain may be interpreted as $\mathsf{Grp}$ if $i=0$
or $\mathsf{Ab}$ if $G$ is commutative).
If the functor $A$ is $H^1(-,G)$ and the functor $H$ is $H^i(-,C)$
for $G$, $C$ algebraic groups, we let
\[
\op{Inv}^i_\ast(G,C)
\]
denote the group of normalized $H$-invariants of $A$.

Let $S$ be a torus.
Recall that $\Ext^1_k(G,S)$ is the group of central extensions of algebraic
groups
\begin{equation} \label{eq:elemOfExt}
1 \to S \to H \to G \to 1 
\end{equation}
up to equivalence under the usual Baer sum.
At the risk of some ambiguity, we will use $[H]$ to denote the class of such an
extension.
Given such an extension, there is a connecting homomorphism
\[
\partial_H : H^1(k,G) \to H^2(k,S)
\]
from the long exact sequence in Galois cohomology.

\begin{theorem} \label{thm:superBM}
Let $G$ be a reductive algebraic group over $k$
and $S$ a special torus over $k$.
Let $\Ext^1_k(G,S)$ be the group of isomorphism classes of central algebraic group
extensions of $G$ by $S$ under Baer sum.
Then the canonical map
\[
\Ext^1_k(G,S) \to \op{Inv}^2_\ast(G,S)
\]
that takes an extension $\xi$
to its connecting homomorphism $\partial_\xi$,
is an isomorphism of groups.
\end{theorem}

\begin{proof}
  For the proof, see the \hyperlink{proof:BM}{Appendix}.
\end{proof}

\begin{remark}
  The above theorem is a generalization of a result of Blinstein and
  Merkurjev~\cite[Theorem 2.4]{Blinstein}, which shows
  that there \emph{exists} an isomorphism when $S=\gm$.
  In their proof, the isomorphism comes from a map in an exact sequence of
  Sansuc (see Proposition~6.10 of \cite{Sansuc1981}), which is somewhat
  mysterious.
  However, Borovoi and Demarche show in \cite[Theorem 2.4]{BorovoiDemarche}
  that one can modify Sansuc's sequence so that it produces exactly the
  isomorphism given in the above theorem.
  Another proof that Blinstein and Merkurjev's result holds with the
  desired isomorphism was given by Lourdeaux in
  \cite[\S{3.1.2}]{Lourdeaux}.
  
  All of the aforementioned results are for the case when $S=\gm$.
  A short argument via Weil restrictions shows that the result holds for
  all special tori.
  The authors had proved Theorem~\ref{thm:superBM} before being made
  aware of the work of Borovoi, Demarche and Lourdeaux.
In addition, a referee pointed out that we can actually
prove our result using only the original weaker statement of Blinstein
and Merkurjev.
  We have retained the theorem above and included its proof in an appendix
since we believe it clarifies matters and is of independent
interest.
\end{remark}

\subsection{Connecting Coflasque Resolutions and Cohomological Invariants}

The purpose of this section is to prove
Theorem~\ref{thm:intro_main_cor}.
We will actually prove a stronger theorem. First, we recall a 
definition.

\begin{definition}
  An algebraic group $G$ over $k$ is \emph{special} if
  $H^1(K,G_K)=\ast$ for every field extension $K/k$.
\end{definition}
For a complete classification of special reductive groups
over an arbitrary field, see \cite{Huruguen}. See also 
\cite{Merk_special} for a refinement, which is more computable in
practice. 

\begin{theorem} \label{thm:big_equivalence}
Let $G$ be a reductive algebraic group $G$ defined over a field $k$.
The following sets are equal:
\begin{enumerate}
\item
$ \displaystyle
\ICP(k,G) = \bigcap_A \ker
\left( H^1(k,G) \to H^1(k,\Aut(A)) \right)
$
where the intersection runs over all separable $k$-algebras $A$
with a $G$-action.
\bigskip

\item
$ \displaystyle
\ICP_{Br}(k,G) := \bigcap_E \bigcap_\alpha \op{ker}
\left( \alpha(k) \colon H^1(k,G) \to \Br(E) \right)
$
where the intersections run over all \'etale algebras $E$
and all normalized cohomological invariants $\alpha$.
\bigskip

\item
$ \displaystyle
\ICP_{qt}(k,G) := \bigcap_S \bigcap_\alpha \op{ker}
\left( \alpha(k) \colon H^1(k,G) \to H^2(k,S) \right)
$
where the intersections run over all quasi-trivial tori $S$
and all normalized cohomological invariants $\alpha$.
\bigskip

\item
$ \displaystyle
\ICP_{sp}(k,G) := \bigcap_S \bigcap_\alpha \op{ker}
\left( \alpha(k) \colon H^1(k,G) \to H^2(k,S) \right)
$
where the intersections run over all special tori $S$
and all normalized cohomological invariants $\alpha$.
\bigskip

\item
$ \displaystyle
\im\left(H^1(k,C) \to H^1(k,G)\right)
$
where $1 \to P \to C \to G \to 1$ is a coflasque resolution
of $G$.
\end{enumerate}
Moreover, if $\bigcap_{i \in I} X_i$ denotes any one of the intersections
from (a) through (d), then there exists an element $i_0 \in I$
such that $X_{i_0} = \bigcap_i X_i$.
\end{theorem}

\begin{remark}
The reductive hypothesis is harmless in
characteristic $0$, since in this case there is a canonical isomorphism
$H^1(k,G) \cong H^1(k,G/U)$ for a connected linear algebraic group
$G$ with unipotent radical $U$.
\end{remark}

\begin{remark}
For a finite constant group $G$, the invariant $\ICP(k,G)$ is always trivial
for almost tautological reasons.
Indeed, $H^1(k,G)$ classifies $G$-Galois algebras over $k$,
which are, in particular, separable algebras with a $G$-action.
\end{remark}

\begin{remark}
We expand on the final statement of the theorem in case (a).
Here, there exists a single separable algebra $A$ with a $G$-action
such that $\Zhe(k,G) = \ker( H^1(k,G) \to H^1(k,\Aut(A)))$.
Note, however, that this algebra $A$ is never unique.
Similar interpretations exist for cases (b)--(d).
\end{remark}

We first reduce Theorem~\ref{thm:big_equivalence} to Theorem~\ref{thm:intro_main_cor} 
following the suggestions of a referee.

\begin{proof}[Proof of equality of (b)-(e) in
Theorem~\ref{thm:big_equivalence}]
It is not hard to see that sets (b)-(d) coincide.
Let us show that the sets (b) and (e) coincide too. Take a coflasque resolution 
\begin{displaymath}
	1 \to P \to C \to G \to 1
\end{displaymath}
and a class $\xi \in \im\left(H^1(k,C) \to H^1(k,G)\right)$. 
Using our assumption of Theorem~\ref{thm:intro_main_cor}, there is
$\zeta \in H^1(k,C)$ mapping to $\xi$.
Take an \'etale algebra $E/k$ with a normalized cohomological invariant 
$\alpha(k) : H^1(k,G) \to \operatorname{Br}(E)$. 

We show that $\alpha(k)(\xi) = 0$. Suppose otherwise. Factor 
\begin{displaymath}
	E = E_1 \times \cdots \times E_m
\end{displaymath}
where each $E_i/k$ is finite and separable. Then we have 
\begin{displaymath}
	\alpha(k)(\xi) = ([D_1],\ldots,[D_m]) 
\end{displaymath}
for central simple algebras $D_i/E_i$.
Without loss of generality $D_1$ is not split.
 
Let us extend scalars by $E_1$.  Since $E_1$ is a direct summand
of $E_1 \otimes_k E_1$, the induced composition
\[
\Br(E_1) \to \Br(E) \to \Br(E \otimes_k E_1) \to \Br(E_1)
\]
factors through the identity on $\Br(E_1)$.
Thus, the composition
\[
H^1(E_1,C_{E_1}) \to H^1(E_1,G_{E_1}) \to \Br(E \otimes_K E_1) \to \Br(E_1)
\]
gives rise to a non-trivial normalized cohomological invariant.
Since $C_{E_1}$ is coflasque, it has trivial Picard group
\cite[Proposition 2.6]{CT2008}.
Thus, by the Blinstein-Merkurjev Theorem~\cite[Theorem 2.4]{Blinstein},
all normalized cohomological invariants are trivial.
This is a contradiction, so $\alpha(k)(\xi)=0$ as desired.
	
For the other direction, take $\xi \in \Zhe^1_{Br}(k,G)$. Suppose we
have a cohomological invariant 
\begin{displaymath}
	\alpha : H^1(-,G) \to H^2(-,P) \cong H^2(-,\operatorname{Br}(E))
\end{displaymath}
where $E/k$ is the \'etale algebra for the quasi-trivial torus $P$.
Since $\alpha(k)(\xi) = 0$, we can lift $\xi$ to some $\zeta$ in $H^1(k,C)$.
\end{proof}

Having reduced to Theorem~\ref{thm:intro_main_cor}, the remainder of
this section is devoted to proving that result. 
From \S{23}~of~\cite{BOI}, we recall some standard facts about
automorphisms of separable algebras.
Let $A$ be a separable algebra over $k$ with center $Z(A)$
(an \'etale algebra over $k$).
Recall that the connected component $\Aut_k(A)^\circ$ of the group
scheme of algebra automorphisms of $\Aut_k(A)$ is the kernel of the restriction
map $\Aut_k(A) \to \Aut_k(Z(A))$.

We have an exact sequence
\begin{equation} \label{eq:GLses}
1 \to \GL_1(Z(A)) \to \GL_1(A) \to \Aut_k(A)^\circ \to 1
\end{equation}
where $\GL_1(B)$ is the group scheme of units of a $k$-algebra $B$
(this is a consequence of the Skolem-Noether theorem).
We define $\PGL_1(A)$ as the quotient
$\GL_1(A)/\GL_1(Z(A)) \cong \Aut_k(A)^\circ$.

\begin{lemma} \label{lem:PGL_from_Brauer}
Suppose there is a central extension of algebraic groups
\[
1 \to S \to H \to G \to 1
\]
and a homomorphism $m : S \to P$ where $P$ is a quasi-trivial torus.
Then there exists a separable algebra $A$ such that $P \cong
\GL_1(Z(A))$
and there is commutative diagram
\[
\SelectTips{cm}{10}\xymatrix{
1 \ar[r] & S \ar[r] \ar[d]^m & H \ar[r] \ar[d] & G \ar[r] \ar[d] & 1 \\
1 \ar[r] & P \ar[r] & \GL_1(A) \ar[r] & \PGL_1(A) \ar[r] & 1 \\
}
\]
with exact rows.
\end{lemma}

\begin{proof}
By taking the pushout of $H$ along $S \to P$,
we may assume that $S=P$ and the morphism $S \to P$ is the identity.

We begin by proving the lemma in the case where $S=\gm$.
The restriction
\[\phi: k[G] \to k[S] = k[t,t^{-1}]\]
is surjective. 
Let $W$ be the inverse image $\phi^{-1}(\langle t \rangle)$
where $\langle t \rangle$ is the $k$-module spanned by $t$.
Since $S$ is central, the natural action of $G$ on $k[G]$ restricts to $W$. 
Any element of $W$ lies in some finite-dimensional 
subrepresentation. Pick some finite-dimensional non-trivial 
sub-representation of $W$ and denote it by $V$. 
The theorem follows, in our current case where $S=P=\gm$, if we 
set $A = \op{End}(V)$. 

We now consider the general case where $S=P$ is quasi-trivial.
It suffices to assume that $P=R_{K/k}\gm$
for a finite separable field extension $K/k$ of degree $n$.
Indeed, quasi-trivial tori are products of such tori;
so the general result follows by taking the product of the constructions.

Let $\pi : \Spec(K) \to \Spec(k)$ be the morphism corresponding to the
field extension $K/k$.
For brevity and clarity we will write $L(X)=X_K$ for $k$-varieties $X$
and $R(Y)=R_{K/k}(Y)$ for $K$-varieties $Y$,
which emphasizes that scalar extension, $L$, is a left adjoint to
Weil restriction $R$.
We have an adjoint pair, so we denote the counit by
$\epsilon : L R \to \id$
and the unit by $\eta : \id \to R L$.

Let $f : RL(\gm)=R_{K/k}(\G_{m,K}) \to H$ be the inclusion of $S$ into $H$.
Define the $K$-group $J$ as the pushout
\[
\xymatrix{
LRL(\gm) \ar[r]^{\epsilon_{L\gm}} \ar[d]^{L f} &
L(\gm)\ar[d]^g  \\
L(H) \ar[r]^h & J }
\]
with $g$ and $h$ the canonical maps.
Since the lemma has been proven for the case $S=\gm$, we have an embedding
$\rho : J \to \GL_{n,K}$ for some $n$ such that $\rho \circ g$ is the identity
on scalar matrices.

We have the following commutative diagram:
\[
\xymatrix{
RL(\gm) \ar[rr]^{\eta_{RL\gm}} \ar[d]^f & &
RLRL(\gm) \ar[rr]^{R\epsilon_{L\gm}}
\ar[d]^{RL f}  &&
RL(\gm) \ar[d]^{R g}  \\
H \ar[rr]^{\eta_H}  &&
RL(H) \ar[rr]^{R h}
& & R(J) }
\]
where the left square commutes due to naturality of $\eta$.
The top row composes to be the identity since
$R\epsilon \circ \eta R = \id$, see e.g. \cite[Theorem IV.1.1]{MacLane}.

Let $A$ be the $k$-algebra of $n \times n$ matrices over $K$.
Since $R_{K/k}(\GL_{n,K})$ is canonically isomorphic to
$\GL_1(A)$,
the composition
\[
R(\rho \circ h) \circ \eta_H : H \to R_{K/k}(\GL_{n,K})
\]
gives the desired map.
The isomorphism $S \to Z(\GL_1(A))$ is given by the top
row of the diagram above.
\end{proof}

With this technical lemma in hand, we are finally able to prove
Theorem~\ref{thm:intro_main_cor}.

\begin{proof}[Proof of Theorem~\ref{thm:intro_main_cor}]
We have already shown that the sets (b) through (e) in
Theorem~\ref{thm:big_equivalence} are equivalent.
Thus we may assume
\[
\ICP_{qt}(k,G) =
\im\left(H^1(k,C) \to H^1(k,G)\right)
\]
where $1 \to P \to C \to G \to 1$ is a coflasque resolution
of $G$.
We will show that
\[
\ICP_{qt}(k,G) \subseteq \ICP(k,G)
\]
and
\[
\ICP(k,G) \subseteq \im\left(H^1(k,C) \to H^1(k,G)\right).
\]

Suppose $x \in \ICP_{qt}(k,G)$.
Let $A$ be an algebra with a group action
$\alpha : G \to \Aut(A)$.
Recall that if $H$ is a linear algebraic group, then
$H^0(-,H) \to H^0(-,\pi_0(H))$ is surjective
by \cite[Theorem 6.5]{Waterhouse}.
Thus, $H^1(-,H^\circ) \to H^1(-,H)$ is injective.
Thus we may assume $\alpha : G \to \Aut(A)^\circ$ instead since
$G$ is connected.
We have a composition
\[
\beta \colon H^1(k,G) \xrightarrow{\partial_{\alpha}}
H^1(k,\Aut(A)^\circ) \hookrightarrow H^2(k,\GL_1(Z(A)))
\]
where the second arrow is injective by Hilbert 90.
In particular, this composition gives rise to a cohomological invariant
and thus $\beta(x)=0$ since $x \in \ICP_{qt}(k,G)$;
thus $\partial_{\alpha}(x)=0$.
We conclude that $x \in \ICP(k,G)$.

Suppose $x \in \ICP(k,G)$.
By Lemma~\ref{lem:PGL_from_Brauer}, 
there exists a separable algebra $A$ such that $P \cong
\GL_1(Z(A))$
and there is commutative diagram
\[
\SelectTips{cm}{10}\xymatrix{
1 \ar[r] & P \ar[r] \ar[d]^= & C \ar[r] \ar[d] & G \ar[r] \ar[d] & 1 \\
1 \ar[r] & P \ar[r] & \GL_1(A) \ar[r] & \PGL_1(A) \ar[r] & 1 \\
}
\]
with exact rows.
Applying Galois cohomology, we observe that the composition
\[
\alpha : H^1(k,G) \to H^1(k, \PGL_1(A)) \to H^2(k,P)
\]
is equal to the connecting homomorphism
\[
\partial : H^1(k,G) \to H^2(k,P) .
\]
Since $\partial(x)=\alpha(x)=0$, we conclude that $x$ is in the image
of $H^1(k,C)$ as desired.
\end{proof}

\begin{remark} \label{rem:refX}
We thank the referee for their insight into simplifying the arguments 
relating Theorems~\ref{thm:intro_main_cor} and~\ref{thm:big_equivalence}.
\end{remark}

%%%%%%%%%%%%%%%%%%%%%%%%%%%%%%%%%%%%%%%%%%%%%%%%%%%%%%%%%%%%%%%%
%%%%%%%%%%%%%%%%%%%%%%%%%%%%%%%%%%%%%%%%%%%%%%%%%%%%%%%%%%%%%%%%
% INTERESTING FIELDS %
%%%%%%%%%%%%%%%%%%%%%%%%%%%%%%%%%%%%%%%%%%%%%%%%%%%%%%%%%%%%%%%%
%%%%%%%%%%%%%%%%%%%%%%%%%%%%%%%%%%%%%%%%%%%%%%%%%%%%%%%%%%%%%%%%

\section{Coflasque algebraic groups over particular fields}
\label{sec:applications}

\subsection{General statements and low cohomological dimension}

We first rephrase the case where $\ICP(-,G)$ is trivial. 

\begin{proposition} \label{prop:ICP_special}
Let $G$ be a reductive group. Then $\ICP(-,G)$ is trivial
if and only if $C$ is special,
where
\[
1 \to P \to C \to G \to 1
\]
is a coflasque resolution of $G$.
\end{proposition}

\begin{proof}
By Theorem~\ref{thm:intro_main_cor}, we have an isomorphism of functors
$\ICP(-,G) \cong H^1(-,C)$; the latter is trivial if and only if $C$ is
special by definition. 
\end{proof}

\begin{proposition}
 Let $T$ be a torus over $k$. Then, $\ICP(-,T)$
 is a stable birational invariant of $T$. Moreover, $\ICP(-,T)$ is trivial
 if and only if $T$ is retract rational.
\end{proposition}

\begin{proof}
Let
\[
1 \to P \to C \to T \to 1
\]
be a coflasque resolution of $T$ of the second type.
Assume we have an exact sequence 
\begin{displaymath}
  1 \to Q \to E \to T \to 1
\end{displaymath}
with $Q$ invertible. 
Taking the fiber product we obtain a commutative diagram
\[
\xymatrix{
& & 1 \ar[d] & 1 \ar[d] \\
& & P \ar@{=}[r] \ar[d] & P \ar[d] \\
1 \ar[r] & Q \ar[r] \ar@{=}[d] & H \ar[d] \ar[r] & C \ar[r] \ar[d] & 1 \\
1 \ar[r] & Q \ar[r] & E \ar[r] \ar[d] & T \ar[r] \ar[d] & 1 \\
& & 1 & 1 & }
\]
with exact rows and columns.
Taking duals of the middle row we get a short exact sequence
\begin{displaymath}
  0 \to \widehat{C} \to \widehat{H} \to \widehat{Q} \to 0
\end{displaymath}
whose associated long exact sequence includes 
\begin{displaymath}
  H^1(\Gamma', \widehat{C}) \to H^1(\Gamma', \widehat{H}) \to H^1(\Gamma', \widehat{Q}) 
\end{displaymath}
for any $\Gamma' \leq \Gamma$. 
Since the outer two terms vanish, so does the middle. 
Thus, $H$ is coflasque. 
Additionally, since $C$ is coflasque and 
$Q$ is quasi-trivial, this extension splits 
$H \cong C \times Q$. Thus, the map 
\begin{displaymath}
  H^1(k,H) \to H^1(k,C) 
\end{displaymath} 
is an isomorphism. Applying Theorem~\ref{thm:intro_main_cor}, we see that 
\begin{displaymath}
  \ICP(k,E) \cong \ICP(k,T). 
\end{displaymath}

Assume that $T$ and $T^\prime$ are stably birational tori. 
Then their flasque invariants coincide 
\cite[Proposition 2.6]{CTS77} and there exist short exact sequences 
\begin{gather*}
 1 \to P \to E \to T \to 1 \\
 1 \to P^\prime \to E \to T^\prime \to 1
\end{gather*}
with both $P$ and $P^\prime$ quasi-trivial 
\cite[Lemme 1.8]{CTS77}. From the above, we see that 
\begin{displaymath}
  \ICP(k,T) \cong \ICP(k,E) \cong \ICP(k,T^\prime). 
\end{displaymath}

Assume that $T$ is retract rational. Then appealing to 
Theorem~\ref{theorem:torus_rationality} we have an exact sequence
\begin{displaymath}
  1 \to Q \to P \to T \to 1
\end{displaymath}
where $Q$ is invertible and $P$ is quasi-trivial. 
Thus, $\ICP(-,T) \cong \ICP(-,P)$. Since $P$ is 
quasi-trivial it is coflasque so $\ICP(-,P) = H^1(-,P)$
by Theorem~\ref{thm:intro_main_cor}. Proposition
~\ref{prop:hilbert90_shapiro} says the latter is trivial. 

Assume $\ICP(-,T)$ is trivial. From Proposition~\ref{prop:ICP_special}, 
$C$ is special. Then, from Corollary~\ref{prop:invertible_H1_vanish}
$C$ is invertible. Then there is a quasi-trivial torus 
$P$ with $P = C \times D$ so 
\begin{displaymath}
  1 \to D \to P \to C \to 1
\end{displaymath}
is a flasque resolution with $D$ invertible. Thus,
Theorem~\ref{theorem:torus_rationality} shows $C$ is retract rational. 
\end{proof}

From Proposition~\ref{prop:ICP_special}, understanding when $\Zhe(k,G)$ is trivial amounts to understanding
when a coflasque algebraic group is special.
When $k$ is perfect and of cohomological dimension $\le 1$, then \emph{all} torsors
of connected algebraic groups are trivial by Serre's Conjecture I
(now Steinberg's Theorem \cite[\S{III.2.3}]{SerreGC}).
Thus, we have:

\begin{proposition}
If $k$ is a field of cohomological dimension $\le 1$,
then $\ICP(k,G)=\ast$ for all reductive algebraic groups $G$.
In particular, this holds for finite fields.
\end{proposition}

In a more subtle manner, we may also leverage Serre's
Conjecture II:

\begin{conjecture}[Serre's Conjecture II]
If $k$ is a perfect field of cohomological dimension $\le 2$,
then $H^1(k,G)=\ast$ for all simply-connected semisimple algebraic groups.
\end{conjecture}

Note that Serre's conjecture II is still open in general, although many
cases are known (see the survey \cite{GilleSurvey}).
In particular, the conjecture is proved for non-archimedean local fields
\cite{Kneser1,Kneser2}.

\begin{proposition} \label{prop:SerreIIapplication}
Suppose $k$ is a field for which the conclusion of Serre's Conjecture II holds.
Let $C$ be a coflasque reductive algebraic group over $k$ and consider the exact
sequence
\[
1 \to C^{sc} \to C \to C^{tor} \to 1
\]
where $C^{sc}$ is the derived subgroup of $C$ and $C^{tor}$ is the
abelianization.
Then the induced map $H^1(k,C) \to H^1(k,C^{tor})$ is injective.
\end{proposition}

\begin{proof}
By definition, the derived subgroup $C^{sc}$ of a coflasque
reductive algebraic group is semisimple simply-connected.
Since any form of a simply-connected semisimple algebraic group
is simply-connected semisimple,
all fibers of the map $H^1(k,C) \to H^1(k,C^{tor})$
are trivial or empty.
\end{proof}

In the remainder of this section, our goal is to
understand $\Zhe(k,G)$ over number fields.
We begin with characterizations of coflasque algebraic groups over local fields.

\begin{lemma} \label{lem:coflasque_local}
If $C$ is a coflasque algebraic group over a nonarchimedean local field $k$,
then $H^1(k,C)=\ast$.
\end{lemma}

\begin{proof}
By Proposition~\ref{prop:SerreIIapplication},
it suffices to assume $C$ is a coflasque torus.
Let $K/k$ be any Galois splitting field of $C$ with Galois group $\Gamma_{K/k}$.
From Tate-Nakayama duality, see e.g. \cite[Theorem 11.3.5]{Vosky}, we
have an isomorphism
\[
H^1(\Gamma_{K/k},C(K)) \cong H^1(\Gamma_{K/k}, \widehat{C}) = 0
\]
since $\widehat{C}$ is coflasque.
Thus $H^1(k,C)=0$ as desired.
\end{proof}

The archimedean case is more complicated.
For real tori, the notions of flasque, coflasque, and
quasi-trivial all coincide, so $H^1(\R,T) = \ast$ for a coflasque real torus $T$.
However, coflasque real algebraic groups can have non-trivial torsors. Thus $\Zhe(\R,G)$ may be non-trivial when $G$ is a not a torus.

\begin{ex}
 The group $\op{SL}_2(\mathbb{H}) \cong \op{Spin}(5,1)$ is simply-connected hence coflasque. However, 
 \begin{displaymath}
   |H^1(\mathbb{R}, \op{SL}_2(\mathbb{H}))| = 2
 \end{displaymath}
 from \cite[Section 10.1]{Adams}. Similarly, for the compact form of $E_8$, which is also simply-connected, we have 
 \begin{displaymath}
   |H^1(\mathbb{R},E_8)| = 3
 \end{displaymath}
 from \cite[Section 10.2]{Adams}.
\end{ex}

Nevertheless, from \cite{Borovoi}, the set $H^1(\R,G)$ has an explicit
combinatorial description for any reductive algebraic group $G$, so this
case can be explicitly computed. 

\subsection{Number fields}
\label{sec:number_fields}

We recall the \emph{Tate-Shafarevich group} of a linear algebraic group
(see, e.g., \cite[\S{7}]{PlatonovRapinchuk}).
If $G$ is a reductive algebraic group and $k$ is a number field, then
 \begin{displaymath}
   \Sha^i(k,G) := \op{ker} \left( H^i(k, G) \to \prod_v
H^i(k_v,G_{k_v}) \right), 
 \end{displaymath}
where the product is over all places $v$ of $k$.
The \emph{Tate-Shafarevich group} is the case where $i=1$,
which is an abelian group even if $G$ is not commutative.

For simply-connected algebraic groups, the Tate-Shafarevich group is trivial.
In fact, we have the following even stronger
result~\cite[Theorem~6.6]{PlatonovRapinchuk}:

\begin{theorem}[Kneser, Harder, Chernousov] \label{thm:KHC}
If $G$ is a simply-connected semisimple algebraic group over a number field $k$,
then the natural map
\[
H^1(k, G) \to \prod_{v \mathrm{\ real}} H^1(k_v,G_{k_v})
\]
is a bijection.
\end{theorem}

\begin{lemma}{\cite[Proposition 9.4(ii)]{CT2008}} \label{lem:Sha1to2}
Let $G$ be a reductive algebraic group over a number field.
Suppose
\[
1 \to S \to H \to G \to 1
\]
is a flasque resolution of $G$.
Then the connecting homomorphism induces
a bijection $\Sha^1(G) \cong \Sha^2(S)$.
\end{lemma}

Finally, we are in a position to prove our final result:

\begin{proof}[Proof of Theorem~\ref{thm:Zhe_number_field}]
Let 
\begin{displaymath}
  1 \to P \to C \to G \to 1
\end{displaymath}
be a coflasque resolution of $G$ and let 
\begin{displaymath}
  1 \to S \to H \to G \to 1
\end{displaymath}
be a flasque resolution of $G$. Setting $H' := C \times_G H$, we obtain
the following commutative diagram with exact rows
\begin{equation} \label{eq:both_resolutions}
\xymatrix{
& & 1 \ar[d] & 1 \ar[d] \\
& & P \ar@{=}[r] \ar[d] & P \ar[d] \\
1 \ar[r] & S \ar[r] \ar@{=}[d] & H' \ar[d] \ar[r] & C \ar[r] \ar[d] & 1 \\
1 \ar[r] & S \ar[r] & H \ar[r] \ar[d] & G \ar[r] \ar[d] & 1 \\
& & 1 & 1 & }
\end{equation}
where $P$ is a quasi-trivial torus, $S$ is a flasque torus,
$H'$ and $H$ are quasi-trivial algebraic groups and $C$ is a coflasque
algebraic group.
By Lemma~\ref{lem:Sha1to2}, the induced maps
\begin{displaymath}
  \Sha^1(k,G) \to \Sha^2(k,S)
\end{displaymath}
and
\begin{displaymath}
  \Sha^1(k,C) \to \Sha^2(k,S)
\end{displaymath}
are isomorphisms.
By Proposition~\ref{prop:flasqueOfCoflasque},
there is a flasque resolution of the first type
\begin{displaymath}
  1 \to S \to \left(H^\prime\right)^{tor} \to C^{tor} \to 1.
\end{displaymath}
Using Lemma~\ref{lem:Sha1to2} again, the induced map
$\Sha^1(k,C^{tor}) \to \Sha^2(k,S)$ is an isomorphism. 
Thus the morphism $\Sha^1(k,C) \to \Sha^1(k,C^{tor})$ is an isomorphism.

The task is to compute $\ICP(k,G)$, Since $\ICP(k,G) \cong H^1(k,C)$ by Theorem~\ref{thm:big_equivalence}, we must compute $H^1(k,C)$.
We start with the short exact sequence
\[
1 \to C^{sc} \to C \to C^{tor} \to 1 \ .
\]
We get a commutative diagram 
\[
  \SelectTips{cm}{10}\xymatrix{
  H^1(k,C^{sc}) \ar[r] \ar[d] & H^1(k,C) \ar[r] \ar[d] & H^1(k,C^{tor}) \ar[d] \\
  {\displaystyle \prod_v} H^1(k_v,C^{sc}) \ar[r] & {\displaystyle \prod_v} H^1(k,C) \ar[r] & {\displaystyle \prod_v} H^1(k_v,C^{tor}) }
\]
From Lemma~\ref{lem:coflasque_local}, we have 
\begin{displaymath}
  H^1(k_v,C) = H^1(k_v,C^{sc}) = H^1(k_v,C^{tor}) = \ast
\end{displaymath}
for any finite $v$; the same holds for complex $v$.
Since coflasque tori are quasi-trivial over $\R$,
we know $H^1(\R,C^{tor})=\ast$.
Thus, we reduce to the commutative diagram
\begin{equation} \label{eq:abelianization_in_disguise}
  \SelectTips{cm}{10}\xymatrix{
  H^1(k,C^{sc}) \ar[r] \ar@{=}[d] & H^1(k,C) \ar[r] \ar[d] & H^1(k,C^{tor}) \ar[d] \\
  {\displaystyle \prod_{v \textrm{ real}}} H^1(k_v,C^{sc}) \ar[r] & {\displaystyle
\prod_{v \textrm{ real}}} H^1(k_v,C) \ar[r] & \ast }
\end{equation}
with exact rows,
where the left vertical map is a bijection by Theorem~\ref{thm:KHC}.
In particular, the map  
\begin{displaymath}
  H^1(k,C) \to \prod_{v \textrm{ real}} H^1(k_v,C) 
\end{displaymath}
is surjective.

We obtain the commutative diagram
\[
\SelectTips{cm}{10}\xymatrix{
1 \ar[r] &
\Sha^1(k,C) \ar[r] \ar@{=}[d] &
H^1(k,C) \ar[r] \ar[d] &
{\displaystyle \prod_{v \textrm{ real}}} H^1(k_v,C) \ar[r] \ar[d] &
\ast
\\
1 \ar[r] &
\Sha^1(k,C^{tor}) \ar@{=}[r] &
H^1(k,C^{tor}) \ar[r] &
1}
\]
with exact rows.
We have a surjective function $H^1(k,C) \to \Sha^1(k,C^{tor})$
that has a canonical retract.

For any cocycle $\gamma \in Z^1(k,C)$, the twisted group
${}^\gamma C^{tor}$ is isomorphic to $C^{tor}$.
Thus $\Sha^1(k,C) \cong \Sha^1(k,{}^\gamma C)$
and we conclude all fibers of the map
\[
H^1(k,C) \to {\displaystyle \prod_{v \textrm{ real}}} H^1(k_v,C)
\]
are isomorphic.
Using Theorem~\ref{thm:intro_main_cor} and the isomorphism $\Sha^1(k,G) \cong \Sha^1(k,C)$, established above, we can rewrite the resulting direct product 
\begin{displaymath}
  \ICP(k,G) \cong H^1(k,C) \cong \Sha^1(k,C) \times \prod_{v \textrm{ real}} H^1(k_v,C) \cong \Sha^1(k,G) \times \prod_{v \textrm{ real}} \ICP(k_v,G).
\end{displaymath}
\end{proof}

\begin{remark}
M.~Borovoi pointed out that Theorem~\ref{thm:Zhe_number_field} can also
be understood using the abelian Galois cohomology group described in
\cite{BorovoiMemoir}.
Indeed, from \cite[Theorem 5.11]{BorovoiMemoir}, the commutative
square
\begin{equation} \label{eq:Borovoi}
 \SelectTips{cm}{10}\xymatrix{
H^1(k,C) \ar[r]^{\op{ab}^1} \ar[d] &
H^1_{\op{ab}}(k,C) \ar[d] \\
{\displaystyle \prod_{v \textrm{ real}} H^1(k_v,C)} \ar[r]^{\op{ab}^1} &
{\displaystyle \prod_{v \textrm{ real}} H^1_{\op{ab}}(k_v,C) }}
\end{equation}
is Cartesian for any reductive group $C$ over a number field.
In the case where $C$ is coflasque,
the maps
$\op{ab}^1 : H^1(-,C) \to H^1_{\op{ab}}(-,C)$ can be identified with the
maps $H^1(-,C) \to H^1(-,C^{tor})$ by \cite[Example 3.14]{BorovoiMemoir}.
One can then show that \eqref{eq:Borovoi} can be identified with the
right hand square from 
\eqref{eq:abelianization_in_disguise}.
\end{remark}

%%%%%%%%%%%%%%%%%%%%%%%%%%%%%%%%%%%%%%%%%%%%%%%%%%%%%%%%%%%%%%%%
%%%%%%%%%%%%%%%%%%%%%%%%%%%%%%%%%%%%%%%%%%%%%%%%%%%%%%%%%%%%%%%%
%%%%%%%%%%%%%%%%%%%%%%%%%%%%%%%%%%%%%%%%%%%%%%%%%%%%%%%%%%%%%%%%
% Appendix %
%%%%%%%%%%%%%%%%%%%%%%%%%%%%%%%%%%%%%%%%%%%%%%%%%%%%%%%%%%%%%%%%
%%%%%%%%%%%%%%%%%%%%%%%%%%%%%%%%%%%%%%%%%%%%%%%%%%%%%%%%%%%%%%%%
%%%%%%%%%%%%%%%%%%%%%%%%%%%%%%%%%%%%%%%%%%%%%%%%%%%%%%%%%%%%%%%%

\appendix
\section{Blinstein and Merkurjev's Theorem}
\label{sec:cohomological_invariants}

In this appendix, we prove Theorem~\ref{thm:superBM}.
As discussed above, this is a generalization of a result of
Blinstein and Merkurjev~\cite[Theorem 2.4]{Blinstein}.
In the case where $S=\gm$, this is proved by Lourdeaux~\cite{Lourdeaux}
or can be proved using \cite[Theorem~2.8]{BorovoiDemarche} by looking at
the details of Blinstein and Merkurjev's proof.
However, we will work with a special torus $S$ throughout.

First, we point out an application.
Namely, every cohomological invariant 
$\alpha : G \to \operatorname{Br}(E)$ arises as a composition of a
projective representation $G \to \operatorname{PGL}_1(E)$ 
via a connecting homomorphism in Galois cohomology. 

\begin{proposition} \label{lem:4flavors}
Let $G$ be a reductive algebraic group $G$ defined over a field $k$.
Let
\[
1 \to P \to C \to G \to 1
\]
be a coflasque resolution of $G$.
For any special torus $S$ and any normalized cohomological invariant
$\alpha \in \op{Inv}^2_\ast(G,S)$,
there exists a group homomorphism $f\colon P \to S$ such that
$\alpha$ is equal to the composite
\[
H^1(k,G) \xrightarrow{\partial_C} H^2(k,P) \xrightarrow{f_\ast} H^2(k,S)
\]
and $\ker(\alpha)$ contains the image of $H^1(k,C) \to H^1(k,G)$.
\end{proposition}

\begin{proof}
By Theorem~\ref{thm:superBM}, every $\alpha$ is obtained as a
connecting homomorphism from some extension
\begin{equation} \label{eq:BMappeal1}
1 \to S \to M \to G \to 1 \ .
\end{equation}
By Proposition~\ref{prop:coflasque_general_versal},
there exists a homomorphism $m : P \to S$ coming from a morphism of
extensions.
Applying Galois cohomology,
we obtain a commutative diagram with exact rows
\[
\SelectTips{cm}{10}\xymatrix{
H^1(k,C) \ar[r] \ar[d] & H^1(k,G) \ar[r] \ar@{=}[d] & H^2(k,P) \ar[d] \\
H^1(k,M) \ar[r] & H^1(k,G) \ar[r]^{\alpha(k)} & H^2(k,S) \\
}
\]
Thus $H^1(k,C)$ is in the kernel of $\alpha$ as desired.
\end{proof}

%%%%%%%%%%%%%%%%%%%%%%%%%%%%%%%%%%%%%%%%%%%%%%%%%%%%%%%%%%%%%%%%
%%%%%%%%%%%%%%%%%%%%%%%%%%%%%%%%%%%%%%%%%%%%%%%%%%%%%%%%%%%%%%%%
% SUBSECTION %
%%%%%%%%%%%%%%%%%%%%%%%%%%%%%%%%%%%%%%%%%%%%%%%%%%%%%%%%%%%%%%%%
%%%%%%%%%%%%%%%%%%%%%%%%%%%%%%%%%%%%%%%%%%%%%%%%%%%%%%%%%%%%%%%%

\subsection{Pairing extensions and torsors}

Let $G$ be a (connected) reductive algebraic group
and $S$ be a special torus.
We define a pairing
\begin{equation} \label{eq:Ext_pairing}
\beta : \Ext^1_k(G,S) \times H^1(k,G) \to H^2(k,S)
\end{equation}
via $\beta([H],[X]) := \partial_H([X])$.
Fixing a $G$-torsor $X \to \Spec(k)$
representing a class in $H^1(k,G)$ we obtain a function
\[
\delta_X : \Ext^1_k(G,S) \to H^2(k,S)
\]
defined by $\delta_X([H])=\partial_H([X])$.

\begin{lemma}
The pairing $\beta$ is additive in the first variable.
In other words, for all $G$-torsors $X \to \Spec(k)$,
the function $\delta_X$ is a group homomorphism. 
\end{lemma}

\begin{proof}
Suppose $H$ and $H'$ are two extensions in $\Ext^1_k(G,S)$.
Let $H''$ denote the Baer sum of $H$ and $H'$.
Recall that this means there is an algebraic group $K$
and a commutative diagram with exact rows
\[
\xymatrix{
1 \ar[r] & S \times S \ar[r] &
H \times H' \ar[r] & G \times G \ar[r] & 1\\
1 \ar[r] & S \times S \ar[r] \ar@{=}[u] \ar[d]^{\mu} &
K \ar[u] \ar[r] \ar[d] & G \ar[u]^{\Delta} \ar[r] \ar@{=}[d] & 1\\
1 \ar[r] & S \ar[r] & H'' \ar[r] & G \ar[r] & 1}
\]
where $\Delta : G \to G \times G$ is the diagonal map
and $\mu : S \times S \to S$ is the multiplication.

Now we compute that
\begin{align*}
& \delta_X([H]) + \delta_X([H']) \\
=& \mu_\ast \left( \delta_X([H]), \delta_X([H']) \right)\\
=& \mu_\ast \left( \partial_H([X]), \partial_{H'}([X]) \right)\\
=& \mu_\ast \left( \partial_{H \times H'}([X \times X]) \right)\\
=& \mu_\ast \left( \partial_{K}([X]) \right)\\
=& \partial_{H''}([X]) = \delta_X([H'']) = \delta_X([H]+[H'])
\end{align*}
where $\mu_\ast : H^2(k,S^2)\to H^2(k,S)$ is induced from multiplication
and each $\partial$ is the connecting homomorphism from $H^1$ to $H^2$
for each exact sequence in the commutative diagram above.
Thus $\delta_X$ is a homomorphism as desired.
\end{proof}

Given a central extension $H$, we obtain a function $\beta(-,[H])$
from $H^1(k,G) \to H^2(k,S)$ that is functorial in $k$.
Thus there is a canonical group homomorphism
\begin{equation} \label{eq:BMmap}
\Ext^1_k(G,S) \to \op{Inv}^2_\ast(G,S) \ .
\end{equation}

%%%%%%%%%%%%%%%%%%%%%%%%%%%%%%%%%%%%%%%%%%%%%%%%%%%%%%%%%%%%%%%%
%%%%%%%%%%%%%%%%%%%%%%%%%%%%%%%%%%%%%%%%%%%%%%%%%%%%%%%%%%%%%%%%
% SUBSECTION %
%%%%%%%%%%%%%%%%%%%%%%%%%%%%%%%%%%%%%%%%%%%%%%%%%%%%%%%%%%%%%%%%
%%%%%%%%%%%%%%%%%%%%%%%%%%%%%%%%%%%%%%%%%%%%%%%%%%%%%%%%%%%%%%%%

\subsection{Torsors over torsors}
\label{sec:torsortorsor}

We recall Corollaire~5.7~from~\cite{CT2008}:

\begin{theorem}[Colliot-Th\'el\`ene] \label{thm:CTtorsor_to_group}
Let $G$ be a connected algebraic group and $S$ an algebraic group of multiplicative
type.
There is an exact sequence
\[
1 \to \Ext^1_k(G,S) \xrightarrow{\psi} H^1(G,S)
\xrightarrow{e^\ast} H^1(k,S) \to 1
\]
where $e : \Spec(k) \to G$ is the inclusion of the identity and
$\psi$ is the ``forgetful map'' that takes the class of a central extension
\[
1 \to S \to H \to G \to 1
\]
to the class of the $S$-torsor $H \to G$.
\end{theorem}

In particular, if $S$ is special, then $\psi$ is an isomorphism.
Thus, if $H \to G$ is an $S$-torsor, then $H$ has a unique
structure of an algebraic group compatible with $G$ and $S$.
We will generalize this and see that an $S$-torsor over a $G$-torsor is
itself a torsor (for an appropriate algebraic group).

\begin{lemma} \label{lem:Rosenlicht_Sansuc}
Let $X$ and $Y$ be smooth varieties over $k$,
with $Y$ separably rational and $Y(k) \neq \emptyset$.  Let $S$ be a special torus.
Then the canonical homomorphism
\[
H^1(X,S) \times H^1(Y,S) \to H^1(X \times Y, S)
\]
is an isomorphism.
\end{lemma}

\begin{proof}
When $S=\gm$, recall that $\Pic(X)=H^1(X,S)$.
Thus the case of $S=\gm$ is exactly \cite[Lemme 6.6]{Sansuc1981}, which
states that the canonical map
\[
\Pic(X) \times \Pic(Y) \to \Pic(X \times Y)
\]
is an isomorphism under the same conditions on $X$ and $Y$.
For a finite separable field extension $L/k$, we have
a canonical isomorphism $H^1(X,S) \cong H^1(X_L,\gm)$
where $S = R_{L/k} \gm$ is the Weil restriction.
Thus, the lemma holds when $S$ is a Weil restriction.
Since there is a canonical isomorphism
\[ H^1(X,S \times S') \cong H^1(X,S) \times H^1(X,S') \]
for tori $S$ and $S'$, the lemma holds for quasi-trivial tori
(since they are precisely the products of Weil restrictions).
Recall that special tori correspond to invertible $\Gamma_k$-modules,
which are direct summands of permutation $\Gamma_k$-modules.
Thus for any special torus $S$, there exists a quasi-trivial
torus of the form $S \times S'$.
Since the composite $S \to S \times S' \to S$ is the identity,
functoriality of $H^1$ shows that the result holds for all special tori. 
\end{proof}

\begin{remark}
The hypothesis that $S$ is special is crucial.
For example, the lemma is false when $k=\R$,
$X = \op{Spec} \mathbb{C}$, $Y = \op{Spec} \mathbb{R}$,
and $S$ is the non-split real one-dimensional torus. 
\end{remark}

We will also require the following amplification of Rosenlicht's Lemma:

\begin{lemma} \label{lem:Rosenlicht}
If $X$ and $Y$ are varieties over $k$ and $S$ is a special torus,
then the canonical homomorphism
$\Hom_k(X,S) \times \Hom_k(Y,S) \to \Hom_k(X \times_k Y,S)$
is an isomorphism.
\end{lemma}

\begin{proof}
Recall that $\Hom_k(X,\gm) = k[X]^\times/k^\times$ is just the group of
invertible regular functions on $X$.
Thus the result for $S=\gm$ is simply Rosenlicht's Lemma \cite[Lemme
10]{CTS77}.
For a finite separable field extension $L/k$,
we have $\Hom_k(X,R_{L/k}\gm) = \Hom_L(X_L,\gm)$.
Thus the result holds for Weil restrictions $S=R_{L/k}\gm$.
Since $\Hom(X,-)$ is additive in the second variable, it applies to
products of Weil restrictions; namely, all quasi-trivial tori $Q$.
Since special tori $S$ possess factorizations $S \to Q \to S$ of the
identity for some quasi-trivial torus, the result holds for all special
tori.
\end{proof}

The following is a variation on \cite[Lemma 2.13]{BorovoiDemarche}:

\begin{theorem} \label{thm:torsors_on_torsors}
Let $G$ be a reductive algebraic group over $k$ and suppose $S$ is a special 
torus over $k$. Let $X \to \op{Spec} k$ be a $G$-torsor,
and let $Y \to X$ be a $S$-torsor.
There exists a central extension
\[
1 \to S \to H \to G \to 1,
\]
unique up to isomorphism of extensions,
along with an $H$-action on $Y$ such that:

\begin{enumerate}
\item the composite $Y \to X \to \op{Spec} k$ is an $H$-torsor,
\item restriction of the $H$-action yields the existing $S$-action on $Y$,
and
\item the induced $H$-action on the quotient $X$ factors through the map
$H \to G$.
\end{enumerate}
\end{theorem}

\begin{proof}
It is well known that every reductive algebraic group is rational over a separably
closed field.  In characteristic $0$, this is due to
Chevalley~\cite{Chevalley} --- in this case, for arbitrary linear algebraic
groups.
However, we could not find a direct reference for
the case of positive characteristic so we sketch a proof that $G$ is
separably rational here. 
From \cite[Exp. XXII Corollary 2.4]{SGA3III}, there is a finite separable extension $L/k$ such that $G_L$ has a split maximal torus $T$ and a system of roots. A set of positive roots and negative roots determines opposite Borels $B,B^\prime$ \cite[Exp. XXII Proposition 5.5.1 and 5.9.2]{SGA3III}.
The unipotent radical $B^u$ of $B$ is isomorphic to an affine space,
and thus $B^\prime  \cong B^{\prime u} \rtimes T$ is rational as well.
The natural map 
\begin{displaymath}
  B^u \times B^\prime \to G_L
\end{displaymath}
is an open immersion \cite[Exp. XXII Proposition 5.9.3]{SGA3III}.
 Hence, $G_L$ is $L$-rational.

Now, let $X \to \op{Spec} k$ be a $G$-torsor and consider the canonical map
\[ \gamma : H^1(X,S) \oplus H^1(G,S) \to H^1(X \times G,S) \]
given by
\[ \gamma(\alpha, \beta)=
\pi^\ast(\alpha) + p^\ast(\beta)
\]
where $\pi : X \times G \to X$
and $p : X \times G \to G$
are the projection maps.
We interpret $\gamma$ geometrically.
Let $Y \to X$ and $H \to G$ be $S$-torsors represented by the classes
$\alpha$ and $\beta$ as above.
Then $\pi^\ast(\alpha)$ represents the $S$-torsor
$Y \times G \to X \times G$ and $p^\ast(\beta)$
represents $X \times H \to X \times G$.
Their sum $\gamma(\alpha,\beta)$ is the quotient of
$(Y \times G) \times_{(X \times G)} (X \times H) \cong Y \times H$
by the diagonal $S$ action.

Note that $X$ is smooth over $\op{Spec} k$. Lemma~\ref{lem:Rosenlicht_Sansuc} says that $\gamma$ is an isomorphism.
Let $q_X : H^1(X \times G, S) \to H^1(X,S)$
and $q_G : H^1(X \times G, S) \to H^1(G,S)$ be the projections
obtained from the inverse of $\gamma$.
Let $\sigma : X \times G \to X$ be the action morphism.
We define
\[ \varphi : H^1(X,S) \to H^1(G,S) \]
as the composition $q_G \circ \sigma^\ast$. 

Let $Y \to X$ be an $S$-torsor.
Then $Z=\sigma^\ast Y$ is an $S$-torsor over $X \times G$.
Since $\gamma$ is an isomorphism,
there exist an $S$-torsor $W \to X$ and
an $S$-torsor $H \to G$ (both unique up to isomorphism)
such that the $(S \times S)$-torsor
$\tau: W \times H \to X \times G$ factors through $Z$ via the diagonal
$S$-action quotient.
In particular, $\varphi([Y])=[H]$.
Let $\iota : X \to X \times G$ be the inclusion via $X \times e_G$.
Since $\sigma \circ \iota$ and $\pi \circ \iota$
are the identity on $X$, we conclude that
$\iota^\ast([Z])=[Y]$ and so $W \cong Y$. 

We can explicitly identify $H \times X$. The torsor of local 
isomorphisms $\op{Iso}_{G \times X}(\pi^\ast Y, \sigma^\ast Y)$ 
realizes the difference $[\sigma^\ast Y] - [\pi^\ast Y]$. Since 
\begin{displaymath}
  \iota^\ast \op{Iso}_{X \times G}(\pi^\ast Y, \sigma^\ast Y)
\cong \op{Iso}_X(Y,Y) \cong S \times X
\end{displaymath}
we know that $e_G^\ast H = S$. Therefore, from Theorem~\ref{thm:CTtorsor_to_group} 
we can endow $H$ with an algebraic group structure sitting in a central extension
\[
1 \to S \to H \to G \to 1
\]
via Theorem~\ref{thm:CTtorsor_to_group}.
We have a commutative diagram
\[
\xymatrix{
Y \times H \ar[r]^{\tau} \ar[d] & Y \ar[d] \\
X \times G \ar[r]^{\sigma} & X}
\]
such that $\tau(ys,hs^\prime)=ss^\prime\tau(y,h)$ for
$y \in Y(\bar{k})$, $h \in H(\bar{k})$, and $s,s^\prime \in S(\bar{k})$.

Note that $\tau$ is not canonical and
may not necessarily be an action map for $H$.
However, let $\chi : Y \to Y$ be the composition
$Y \times \{e\} \to Y \times H \to Y$, which is a morphism of
$S$-torsors over $X$.
We replace $\tau$ with $\chi^{-1} \circ \tau$ and claim that now $\tau$
is a group action.

To check that $\tau$ is a group action,
it suffices to assume $k$ is algebraically closed.
By the modification above, we have $\tau(y,e_H)=y$ for $y \in Y(\bar{k})$.

Now $\sigma : X \times G \to X$ is a right action and thus
there exists a homomorphism $\omega: Y \times H \times H \to S$
factoring through $X \times G \times G$ such that
\begin{equation} \label{eq:putative_action}
\tau(y,h_1h_2) = \tau(\tau(y,h_1),h_2) \omega(y,h_1,h_2)
\end{equation}
for $y \in Y$ and $h_1,h_2 \in H$.
By Lemma~\ref{lem:Rosenlicht},
\[ \omega(y,h_1,h_2)=
\chi_1(y)
\chi_2(h_1)
\chi_3(h_2) \]
where $\chi_1 : Y \to S$ is a map factoring through $X$
and $\chi_2, \chi_3 : H \to S$ are morphisms
factoring through $G$.

Taking $h_1=h_2=e_H$ in \eqref{eq:putative_action}, we find
\[
y=y\chi_1(y)\chi_2(e_H)\chi_3(e_H),
\]
which shows that $\chi_1$ is a constant function.
Taking $h_1=e_H$ in \eqref{eq:putative_action}, we find
\[
\tau(y,h)=\tau(y,h)\chi_1(y)\chi_2(e_H)\chi_3(h),
\]
which shows that $\chi_3$ is a constant function.
Similarly, taking $h_2=e_H$ shows that $\chi_2$ is a constant function.
Thus $\omega$ is a constant function.
Since $\omega(y,e_H,e_H)=e_S$ and $\omega$ is constant,
we conclude that $\tau$ is an action.
\end{proof}

%%%%%%%%%%%%%%%%%%%%%%%%%%%%%%%%%%%%%%%%%%%%%%%%%%%%%%%%%%%%%%%%
%%%%%%%%%%%%%%%%%%%%%%%%%%%%%%%%%%%%%%%%%%%%%%%%%%%%%%%%%%%%%%%%
% SUBSECTION %
%%%%%%%%%%%%%%%%%%%%%%%%%%%%%%%%%%%%%%%%%%%%%%%%%%%%%%%%%%%%%%%%
%%%%%%%%%%%%%%%%%%%%%%%%%%%%%%%%%%%%%%%%%%%%%%%%%%%%%%%%%%%%%%%%

\subsection{An exact sequence of Sansuc}
\label{sec:sansuc}

Given a $G$-torsor $\pi : X \to Y$ with $Y$ smooth, a long exact sequence is
constructed by Sansuc in \cite[Proposition~6.10]{Sansuc1981},
which contains the important subsequence:
\[
\Pic(X) \xrightarrow{\varphi} \Pic(G) \to \Br(Y)
\xrightarrow{\pi^\ast} \Br(X).
\]
However, the map $\Pic(G) \to \Br(Y)$ is constructed by a series of maps
obtained from spectral sequences and thus is somewhat obscure.
For our applications, we need to know a specific interpretation for
this map.

The following theorem can be seen as an expansion of Sansuc's result, which extends $\gm$ to a general
special torus $S$ and explicitly describes the maps occurring in the
sequence over the base $\op{Spec} k$. We remind the reader that the map $\delta_X: H^1(G, S) \to H^2(k, S)$ is defined by $\delta_X([H])=\partial_H([X])$.

\begin{theorem} \label{thm:sansuc_with_delta}
Let $G$ be a reductive algebraic group.
Let $S$ be a special torus and suppose $\pi : X \to \op{Spec} k$ is a $G$-torsor.
Then the sequence
\begin{equation} \label{eq:sansuc}
H^1(X,S) \xrightarrow{\varphi}
H^1(G,S) \xrightarrow{\delta_X}
H^2(k,S) \xrightarrow{\pi^\ast}
H^2(X,S)
\end{equation}
is exact.
\end{theorem}

When $S=\gm$ and the characteristic is $0$, this result is a special case of
\cite[Lemma 2.13]{BorovoiDemarche}.

\begin{proof}
\textbf{Claim:} The composite
$H^1(X,S) \xrightarrow{\varphi}
H^1(G,S) \xrightarrow{\delta_X}
H^2(k,S)$
is trivial.

Let $Y \to X$ be an $S$-torsor.
By Theorem~\ref{thm:CTtorsor_to_group},
there is an algebraic group $H$ representing $\varphi([Y])$ in
$H^1(G,S)$.  By construction, $Y \to \op{Spec} k$ is an $H$-torsor
whose image under $H^1(k,H) \to H^1(k,G)$ is the isomorphism class of
$X$.  In particular, $\partial_H([X])=0$.
We have $\delta_X([H])=\partial_H([X])$, so the claim follows.

\bigskip

\textbf{Claim:} 
$\im(\varphi) = \ker(\delta_X)$.

Suppose $H$ is a central extension of $G$ by $S$ such that
$\delta_X([H])=0$.  Then $\partial_H([X])=0$.
This implies there exist an $H$-torsor $Y$ and
a $(H \to G)$-equivariant map $Y \to X$.
Thus $Y \to X$ is an $S$-torsor and we conclude $\varphi([Y])=[H]$.

\bigskip

\textbf{Claim:} 
The composite
$H^1(G,S) \xrightarrow{\delta_X}
H^2(k,S) \xrightarrow{\pi^\ast}
H^2(X,S)$
is trivial.

Let $H$ be a group extension representing an element in
$H^1(G,S)$.
Recall from \cite[IV.3.4.1]{Giraud}
that $\partial_H([X])$ can be interpreted as (the equivalence
class of) the $S$-gerbe $\mathcal{G}$ of lifts of $X$ to $H$.
Specifically, $\mathcal{G}(U)$ is the category of
$H$-torsors $Y \to U$ with
$(H \to G)$-equivariant maps $Y \to X_U$.
To prove the lemma, we want to show the pullback of $\mathcal{G}$ along
$X \to \Spec(k)$ is a trivial gerbe.
This is equivalent to showing that $\mathcal{G}(X)$ is non-empty.
Thus, we need to find an $H$-torsor $Y \to X$
with an $(H \to G)$-equivariant map $Y \to X \times X$.
Of course, there is an isomorphism $X \times G \cong X \times X$
by the definition of a $G$-torsor, so the composite
$X \times H \to X \times G \cong X \times X$
is the desired map.

\bigskip

\textbf{Claim:}
$\im(\delta_X) = \ker(\pi^\ast)$.

Let $\mathcal{C}$ be an $S$-gerbe over $\Spec(k)$ such that
$p^\ast \mathcal{C}$ is trivial.
Then there exists an $S$-torsor $W \to X$ in $\mathcal{C}(X)$.
Let $H$ be the extension of $G$ by $S$ that acts on $W$
given in Theorem~\ref{thm:torsors_on_torsors}.
Let $\mathcal{G}$ be the $S$-gerbe of lifts of $X$ to $H$
as in the previous claim.
We will construct a morphism $\Upsilon : \mathcal{C} \to \mathcal{G}$
of $S$-gerbes, which is then automatically an equivalence.
This then proves the claim.

For a fixed base $U$ and $S$-torsors $A,B$,
let $\op{Iso}^S_U(A,B)$ denote the $U$-scheme of morphisms of
$S_U$-torsors $A \to B$.  Note that $\op{Iso}^S_U(A,B)$ is itself an
$S_U$-torsor over $U$ via pre- or post-composition.
An element $Z \in \mathcal{C}(U)$ is an $S$-torsor $Z \to U$.
Note that $W \times U$ and $X \times Z$ are canonically $S$-torsors
over $X \times U$.
Thus we can construct the $S$-torsor
\[
T_Z := \op{Iso}^S_{X \times U}(W \times U, X \times Z)
\]
over $X_U = X \times U$.
Note that $T_Z$ has an $H_U$-action via pre-composition.
Moreover, the composite $T_Z \to X_U \to U$ is
the quotient map for the $H_U$-action.
In fact, $T_Z \to U$ is an $H$-torsor with an $(H \to G)$-equivariant map
$T_Z \to X_U$.
In other words, $T_Z$ is an object in $\mathcal{G}(U)$.

The construction of $T_Z$ is functorial in $Z$,
and so we have a functor $\Upsilon : \mathcal{C} \to \mathcal{G}$
such that $\Aut_{\mathcal{C}}(Z) \to \Aut_{\mathcal{G}}(T_Z)$
induces the identity on $S$ after the canonical identifications.
We conclude that we have a morphism (thus an equivalence) of $S$-gerbes as
desired.
\end{proof}

%%%%%%%%%%%%%%%%%%%%%%%%%%%%%%%%%%%%%%%%%%%%%%%%%%%%%%%%%%%%%%%%
%%%%%%%%%%%%%%%%%%%%%%%%%%%%%%%%%%%%%%%%%%%%%%%%%%%%%%%%%%%%%%%%
% SUBSECTION %
%%%%%%%%%%%%%%%%%%%%%%%%%%%%%%%%%%%%%%%%%%%%%%%%%%%%%%%%%%%%%%%%
%%%%%%%%%%%%%%%%%%%%%%%%%%%%%%%%%%%%%%%%%%%%%%%%%%%%%%%%%%%%%%%%

\subsection{Blinstein and Merkurjev}
\label{sec:BM}

\begin{lemma} \label{lem:BrFunctionField}
Let $X$ be a regular variety over $k$
and let $S$ be a special torus over $k$.
Then the homomorphism obtained by the pullback of the generic point
\[
H^2(X,S) \to H^2(k(X),S)
\]
is injective.
\end{lemma}

\begin{proof}
The case of $S=\gm$ is simply that case of Brauer groups
$\Br(X)\to \Br(k(X))$, which is standard
(see, e.g., \cite[Example 2.22]{Milne80}).
Thus, for any \'etale $k$-algebra $L$,
we have the injectivity of $H^2(X_L,\gm) \to H^2(L \otimes_k k(X),\gm)$
and thus injectivity of $H^2(X,Q) \to H^2(k(X),Q)$ for
all quasi-trivial tori $Q=R_{L/k}\gm$.
For a special torus $S$, the result follows from the
factorization $S \to Q \to S$ of the identity
through a quasi-trivial torus.
\end{proof}

The following is a mild generalization of \cite[Lemma 2.3]{Blinstein}.

\begin{lemma} \label{BM:Pic_stable}
Let $G$ be an algebraic group
and let $S$ be a special torus over $k$,
and suppose $K/k$ is a field extension such that $k$ is algebraically closed in
$K$.
Then the natural map $H^1(G,S) \to H^1(G_K,S_K)$ is an isomorphism.
\end{lemma}

\begin{proof}
When $S=\gm$, this result is \cite[Lemma 2.3]{Blinstein},
which incidentally uses both a coflasque resolution of $G$
and a piece of \cite[Proposition 6.10]{Sansuc1981}
in the proof.  Thus we may assume the lemma holds for $S=\gm$.

Consider a finite separable extension $F/k$,
and let $S=R_{F/k}\gm$.
Note that $FK = F \otimes_k K$ is a field since $k$ is algebraically closed
in $K$.
From the $\gm$ case of the lemma, the natural map
\[
H^1(G_F,\G_{m,F}) \to H^1(G_{FK},\G_{m,FK})
\]
is an isomorphism.
By the Weil restriction adjunction, we see that
\[
H^1(G,R_{F/k}\G_{m,F}) \to H^1(G_{K},R_{FK/K}\G_{m,FK})
\]
is an isomorphism.
Note that $R_{FK/K}\G_{m,FK}$ is canonically isomorphic to
$(R_{F/k}\G_{m,F})_K$ since they both represent the functor
\[
A \mapsto ((F \otimes_k K) \otimes_K A)^\times
\]
on $K$-algebras $A$.
Thus,
the natural map $H^1(G,S) \to H^1(G_K,S_K)$ is an isomorphism
in the case where $S=R_{F/k}\gm$.

Since the functor $H^1(G,-)$ preserves finite limits,
the result holds when $S$ is quasi-trivial.
If $S$ is a special torus, then there is a quasi-trivial torus $Q$
along with morphisms $S \to Q \to S$ that compose to the identity.
The result now follows by functoriality of $H^1(G,-)$.
\end{proof}

\begin{proof}[\hypertarget{proof:BM}{Proof of Theorem~\ref{thm:superBM}}]
The following is adapted from the proof of
\cite[Theorem 2.4]{Blinstein}.
Note that this theorem has been proved in the case $S=\gm$ by
Lourdeaux~\cite{Lourdeaux}.

Recall that we want to show that the map
\[
\Ext^1_k(G,S) \to \op{Inv}^2_\ast(G,S),
\]
which takes an extension $\xi$ to its connecting homomorphism
$\partial_\xi$, is a group isomorphism.
Precomposing with the canonical identification
$H^1(G,S) \cong \Ext^1_k(G,S)$ we obtain
\[
\nu: H^1(G,S) \to \op{Inv}^2_\ast(G,S),
\]
which we will prove is an isomorphism.

The remainder of the proof makes use of \emph{versal torsors}
--- see \cite[Section 5]{Skip}.
Since $G$ is a linear algebraic group, there exists an embedding
of algebraic groups $G \to \GL_n$ for some $n$.
The quotient $\GL_n \to \GL_n/G$ is a $G$-torsor
and the pullback by the generic point $\pi \colon T \to \Spec(K)$
is a versal $G$-torsor.
Consider the map
\[
\theta : \op{Inv}^2_\ast(G,S) \to H^2(K,S)
\]
that sends a cohomological invariant $\alpha$ to its value
$\alpha(T)$ for the versal torsor $T \to \Spec(K)$.
By \cite[Theorem 12.3]{Skip}, the map $\theta$ is injective.

We claim $H^1(T,S)=0$.
From \cite[D\'efinition 1.8]{CT2008}, we recall that a
geometrically-integral variety $X$ over $k$ is \emph{finie-factorielle}
if $\Pic(X_K)=0$ for all finite separable field extensions $K/k$.
From \cite[Proposition 1.9]{CT2008}, if $X$ is smooth and
finie-factorielle, then so is every open subset.
In particular, $\GL_{n,k}$ is finie-factorielle since it is an open subset of
affine space.
From \cite[Proposition 1.10]{CT2008}, $H^1(X,Q)=0$ for every
finie-factorielle $X$ and quasi-trivial torus $Q$.
Thus $H^1(U,S)=0$ for any open subset $U$ of $\GL_n$ since $S$ is a
direct multiplicand of some such $Q$.
From \cite[\href{https://stacks.math.columbia.edu/tag/09YQ}{Tag 09YQ}]{stacks-project},
we conclude
\[
H^1(T,S) = \colim_U H^1(U,S) = 0
\]
where the limit is over all open subsets $U$ of $\GL_n$ containing $T$.

An element $H \in H^1(G,S)$ can be interpreted as a group extension
\[
1 \to S \to H \to G \to 1
\]
and therefore $\nu(H)(K)$ is the connecting homomorphism
$\partial_{H_K} : H^1(K,G_K) \to H^2(K,S_K)$;
in particular, $(\theta \circ \nu)(H)=\partial_{H_K}(T)$.
Let $j : H^1(G,S) \to H^1(G_K,S_K)$ be the isomorphism
from Lemma~\ref{BM:Pic_stable}.
We see that
$(\delta_T \circ j)(H) = \delta_{T}(H_K)$,
which is equal to $\partial_{H_K}(T)$ by the definition of $\delta_T$.
Thus, there is a commutative diagram
\[
\xymatrix{
& H^1(G,S) \ar[r]^\nu \ar[d]^j &
\op{Inv}^2_\ast(G,S) \ar[d]^{\theta} \\
H^1(T,S_K) \ar[r] &
H^1(G_K,S_K) \ar[r]^{\delta_T} &
H^2(K,S_K) \ar[r]^{\pi^\ast} &
H^2(T,S_K)
}
\]
where the bottom sequence is \eqref{eq:sansuc}
for the $G_K$-torsor $T$.
Since $H^1(T,S_K)=0$ and $j$ is an isomorphism, we see that $\delta_T$ is injective.

The pullback map $i : H^2(T,S) \to H^2(k(T),S)$ is injective
by Lemma~\ref{lem:BrFunctionField}.
The composite
\[
i \circ \pi^\ast \circ \theta \colon
\op{Inv}^2_\ast(G,S) \to H^2(k(T),S)
\]
takes a cohomological invariant $\alpha \in \op{Inv}^2_\ast(G,S)$ to
$\alpha\left( T \times_K \Spec(k(T)) \right)$
since $\alpha$ is a natural transformation of functors from $k\hyph\mathsf{Fld}$ to $\mathsf{Set}_\ast$.
Note that the generic point lifts to a rational point of the
torsor $T_{\Spec(k(T))}$.
Thus, the torsor $T$ is trivialized by $\Spec(k(T))$ and we conclude
that $\im(\theta) \subseteq \ker(\pi^\ast)$.
It follows that $\nu$ is an isomorphism. 
\end{proof}

%%%%%%%%%%%%%%%%%%%%%%%%%%%%%%%%%%%%%%%%%%%%%%%%%%%%%%%%%%%%%%%%
%%%%%%%%%%%%%%%%%%%%%%%%%%%%%%%%%%%%%%%%%%%%%%%%%%%%%%%%%%%%%%%%
%%%%%%%%%%%%%%%%%%%%%%%%%%%%%%%%%%%%%%%%%%%%%%%%%%%%%%%%%%%%%%%%
% BIBLIOGRAPHY %
%%%%%%%%%%%%%%%%%%%%%%%%%%%%%%%%%%%%%%%%%%%%%%%%%%%%%%%%%%%%%%%%
%%%%%%%%%%%%%%%%%%%%%%%%%%%%%%%%%%%%%%%%%%%%%%%%%%%%%%%%%%%%%%%%
%%%%%%%%%%%%%%%%%%%%%%%%%%%%%%%%%%%%%%%%%%%%%%%%%%%%%%%%%%%%%%%%

\bibliographystyle{halpha-abbrv}

\providecommand\noopsort[1]{}

\end{document}